\documentclass[12pt,reqno]{amsart}
\usepackage{amsmath,amssymb,amsfonts,amsthm,mathrsfs}
\usepackage[foot]{amsaddr}
\usepackage[left=3cm,top=3cm,right=3cm,bottom=3cm]{geometry}
\usepackage{epsfig}
\usepackage{graphicx}
\usepackage[usenames,dvipsnames]{xcolor}
\usepackage{tikz} \usetikzlibrary{calc, backgrounds, fit}
\usepackage{soul}

\usepackage{biblatex}
\bibliography{subhypergraphs-final}

\newcommand{\ppc}[1]{{\color{Blue}{\bf [~Pawel:\ }{\em #1}~{\bf ]}}}

\DeclareMathOperator{\aut}{aut}
\DeclareMathOperator{\Var}{\mathbb Var}
\DeclareMathOperator{\Cov}{\mathbb Cov}
\DeclareMathOperator{\Prob}{\mathbb{P}}

\newenvironment{hypergraph}[1][]{
\begin{tikzpicture}[#1]
\tikzstyle{vertex} = [fill=black, circle, inner sep=0pt, minimum size=8pt]
\tikzstyle{edge} = [fill,opacity=.5,fill opacity=0,line cap=round, line join=round, line width=25pt]
}{\end{tikzpicture}}

\newtheorem{theorem}{Theorem}[section]

\newtheorem{proposition}[theorem]{Proposition}

\newtheorem{corollary}[theorem]{Corollary}

\newcommand{\mbf}[1] {\text{\boldmath$#1$}}

\newcommand{\ppp}{\mbf{p}}
\newcommand{\cF}{{\mathcal F}}

\newcommand{\cJ}{{\mathcal J}}
\newcommand{\N}{{\mathbb N}}
\newcommand{\Z}{{\mathbb Z}}

\newcommand{\E}{\mathbb E}
\newcommand{\ev}{\mbf e}
\newcommand{\eps}{\varepsilon}
\newcommand{\Hh}{{\mathscr{H}}}
\newcommand{\sH}{{\mathscr{H}}}
\newcommand{\G}{{\mathscr{G}}}

\author{Megan Dewar$^1$, John Healy$^1$, Xavier P\'{e}rez-Gim\'{e}nez$^2$, Pawe\l{} Pra\l{}at$^3$, John Proos$^1$, Benjamin Reiniger$^4$, Kirill Ternovsky$^3$}

\address{$^1$ The Tutte Institute for Mathematics and Computing, Ottawa, ON, Canada}
\address{$^2$ Department of Mathematics, University of Nebraska, Lincoln, NE, USA}
\address{$^3$ Department of Mathematics, Ryerson University, Toronto, ON}
\address{$^4$ Department of Applied Mathematics, Illinois Institute of Technology, Chicago, IL, USA}

\email{tutte.institute+MeganDewar@gmail.com, tutte.institute+JohnHealy@gmail.com}
\email{xperez@unl.edu, pralat@ryerson.ca, tutte.institute+JohnProos@gmail.com, breiniger@iit.edu}
\email{kirill.ternovsky@ryerson.ca}

\keywords{random graphs, random hypergraphs, subgraphs, subhypergraphs}

\begin{document}

\title{Subhypergraphs in non-uniform random hypergraphs}

\begin{abstract}
In this paper we focus on the problem of finding (small) subhypergraphs in a (large) hypergraph. We use this problem to illustrate that reducing hypergraph problems to graph problems by working with the 2-section is not always a reasonable approach.  We begin by defining a generalization of the binomial random graph model to hypergraphs and formalizing several definitions of subhypergraph.  The bulk of the paper focusses on determining the expected existence of these types of subhypergraph in random hypergraphs.  We also touch on the problem of determining whether a given subgraph appearing in the 2-section is likely to have been induced by a certain subhypergraph in the hypergraph.  To evaluate the model in relation to real-world data, we compare model prediction to two datasets with respect to (1) the existence of certain small subhypergraphs, and (2) a clustering coefficient.
\end{abstract}

\maketitle

\section{Introduction}\label{sec:intro}

Myriad problems can be described in hypergraph terms, however, the theory and tools are not sufficiently developed to allow most problems to be tackled directly within this context.  In particular, we lack even the most basic of hypergraph models.  In this paper we introduce a natural generalization of Erd\H{o}s-R\'enyi (binomial) random graphs to non-uniform random hypergraphs.  While such a model cannot hope to capture many features of real-world datasets, it allows us to explore several fundamental questions regarding the existence of subhypergraphs and helps us to illustrate that the common practice of reducing hypergraph problems to graph problems via the 2-section operation is not reasonable in many cases.

Despite being formally defined in the 1960s (and various realizations studied long before that) hypergraph theory is patchy and often not sufficiently general.  The result is a lack of machinery for investigating hypergraphs, leading researchers and practitioners to create the 2-section graph of a hypergraph of interest and then rely upon well-established graph theoretic tools for analysis.  In taking the 2-section (that is, making each hyperedge a clique, see Section~\ref{sec:definitions} for a formal definition) we lose some information about edges of size greater than two.  Sometimes losing this information does not affect our ability to answer questions of interest, but in other cases it has a profound impact.

Let us further explore the fundamental issues with working on the 2-section graph by considering a small example.  Suppose we are given the coauthorship hypergraph in which vertices correspond to researchers and each hyperedge consists of the set of authors of a scientific paper.  We wish to answer two questions regarding this dataset: (1) What is the Erd\H{o}s number of every researcher (zero for  Erd\H{o}s, one for coauthors of Erd\H{o}s, two for coauthors of coauthors of Erd\H{o}s, etc.)? and (2) Given a subhypergraph induced by only the seminal papers in a particular field, what is the minimum set of authors who between them cover all the papers in the subhypergraph? The Erd\H{o}s number of an author is the minimum distance between the author's vertex and Erd\H{o}s' vertex in the hypergraph and this distance is not changed by taking the 2-section.  However, in the case of finding a minimum set of vertices that are incident with every hyperedge in the subhypergraph, the 2-section of the hypergraph loses the information about the set of papers that a particular author covers.  In fact, the 2-section does not even retain how many papers were used to create the hypergraph!  Fundamentally, if  the composition of the hyperedges of size greater than two is important in solving a problem, then solving the problem in the 2-section will be difficult, if not impossible.

Besides the information loss, there is another potential downside to working with the 2-section of a hypergraph; the 2-section can be much denser than the hypergraph since a single hyperedge of size $k$ implies $\binom{k}{2}$ edges in the 2-section. Depending on the dataset and algorithm being executed, the increased density of the 2-section can have a significant detrimental effect on compute time.  

In this paper we are interested in finding subhypergraphs in hypergraphs. We study rigorously -- via theorems with proofs -- occurrences of a given hypergraph as a subhypergraph of random hypergraphs generated by our model.  One of the implications of our work is that two hypergraphs $H_1$ and $H_2$ that induce the same subgraph in the 2-section can have drastically different thresholds for appearance. This is concrete evidence that the research community needs to develop more algorithms that deal with hypergraphs directly.  

To evaluate our model and, in particular, to illuminate features of real-world networks not captured by the model, we investigate two datasets: an email hypergraph and a coauthorship hypergraph. Not surprisingly, we confirm that subhypergraphs that are not distinguishable in the 2-section graph occur with different probabilities (as predicted by the model). However, also not surprisingly, we confirm that the distribution of such subhypergraphs in the two networks is quite different from what the model predicts.  This is due in large part to the fact that the model predicts that edges occur independently.  In graph theory the interdependence of edges is measured by the notion of clustering coefficient. There have been a number of proposals for generalizing clustering coefficient to hypergraphs including \cite{AR}, \cite{BE}, \cite{GLL}, \cite{LMV}, and \cite{ZN}. We calculate the hypergraph clustering coefficient from \cite{ZN} for our random hypergraph model and the two real networks we are investigating.  This positions us well for further model development.
An extended abstract of this paper appeared in~\cite{WAW_version}.

\section{Definitions and Conventions}\label{sec:definitions}

\subsection{Random graphs and random hypergraphs} 

First, let us recall a classic random graph model. A \emph{binomial random graph} $G \in \G(n,p)$ is a random graph with vertex set $[n] :=\{1,2,\dots,n\}$ in which every pair $\{i,j\} \in \binom{[n]}{2}$ appears independently as an edge with probability~$p$. Note that $p=p(n)$ may (and usually does) tend to zero as $n$ tends to infinity.  

In this paper, we are concerned with a more general combinatorial object: hypergraphs. A \emph{hypergraph} $H$ is an ordered pair $H=(V,E)$, where $V$ is a finite set (the \emph{vertex set}) and $E$ is a family of subsets of $V$ (the \emph{hyperedge set}). A hypergraph $H=(V,E)$ is \emph{$r$-uniform} if all hyperedges of $H$ are of size $r$. For a given $r \in \N$, a \emph{random $r$-uniform hypergraph} $H \in \Hh_r(n,p)$ has $n$ labelled vertices from a vertex set $V = [n]$, with every subset $e \subseteq V$ of size $|e|=r$ chosen to be a hyperedge of $H$ randomly and independently with probability $p$. For $r=2$, this model reduces to the model $\G(n,p)$. 

The binomial random graph model is well known and thoroughly studied (e.g.~\cite{bol,JLR, Book}). Random hypergraphs are much less well understood and  most of the existing papers deal with uniform hypergraphs (e.g. Hamilton cycles (both tight ones and loose ones) were recently studied in~\cite{Book252, Book253, Book292}, perfect matchings were investigated in~\cite{Book451} and additional examples can be found in a recent book on random graphs~\cite{Book}). 

In this paper, we study a natural generalization of the $r$-uniform random hypergraph model which produces non-uniform hypergraphs. Let $\ppp = (p_r)_{r \ge 1}$ be any sequence of numbers such that $0 \le p_r = p_r(n) \le 1$ for each $r \ge 1$. A \emph{random hypergraph} $H \in \Hh(n,\ppp)$ has $n$ labelled vertices from a vertex set $V = [n]$, with every subset $e \subseteq V$ of size $|e|=r$ chosen to be a hyperedge of $H$ randomly and independently with probability $p_r$. In other words, $\Hh(n,\ppp) = \bigcup_{r \ge 1} \Hh_r(n,p_r)$ is a union of independent uniform hypergraphs. 

Let us mention that there are several other natural generalizations that might be worth exploring, depending on the specific application in mind.  One possible generalization would be to allow hyperedges to contain repeated vertices (multiset-hyperedge hypergraphs). Another would be to allow the hyperedges to be chosen with possible repetition, resulting in parallel hyperedges.  We do not address these alternative formulations in this paper.

We require a few other definitions to aid our discussions.  A vertex of a hypergraph is \emph{isolated} if it is contained in no edge.  In particular, a vertex of degree one that belongs only to an edge of size one is not isolated.
The \emph{2-section} of a hypergraph $H$, denoted $[H]_2$, is the graph on the same vertex set as $H$ and an edge $\{u,v\}$ if (and only if) $u$ and $v$ are contained in some edge of $H$.  In other words, the 2-section is obtained by making each hyperedge of $H$ a clique in $[H]_2$. The \emph{complete hypergraph} on $n$ vertices is the hypergraph with all $2^n-1$ possible nonempty edges. 

\subsection{Notation}

All asymptotics throughout are as $n$ goes to $\infty$.  We emphasize that the notations $o(\cdot)$ and $O(\cdot)$ refer to functions of $n$, not necessarily positive, whose growth is bounded. We also use the notation $f \ll g$ for $f=o(g)$ and $f \gg g$ for $g=o(f)$. We say that an event in a probability space parametrized by $n$ holds \emph{asymptotically almost surely} (or \emph{a.a.s.}) if the probability that it holds tends to $1$ as $n \to \infty $. Since we aim for results that hold a.a.s., we will always assume that $n$ is large enough. We will often abuse notation by writing $\G(n,p)$ or $\Hh(n,\ppp)$ to refer to a graph or hypergraph drawn from the distributions $\G(n,p)$ and $\Hh(n,\ppp)$, respectively.  For simplicity, we will write $f(n) \sim g(n)$ if $f(n)/g(n) \to 1$ as $n \to \infty$ (that is, when $f(n) = (1+o(1)) g(n)$). Finally, we use $\log n$ to denote natural logarithms.

\subsection{Subhypergraphs}

In this paper, we are concerned with occurrences of a given substructure in hypergraphs. As there are at least two natural generalizations of ``subgraph'' to hypergraphs, we cannot simply call these substructures ``subhypergraphs''.

A hypergraph $H' = (V', E')$ is a \emph{strong subhypergraph} (called \emph{hypersubgraph} by Bahmanian and Sajna~\cite{BS} and \emph{partial hypergraph} by Duchet~\cite{duc}) of $H = (V, E)$ if $V' \subseteq V$ and $E' \subseteq E$; that is, each hyperedge of $H'$ is also an hyperedge of $H$.  We write $H'\subseteq_s H$ when $H'$ is a strong subhypergraph of $H$.
For $H=(V,E)$ and $V'\subseteq V$, the \emph{strong subhypergraph of $H$ induced by $V'$}, denoted $H_s[V']$, has vertex set $V'$ and hyperedge set $E' = \{ e \in E : e \subseteq V' \}$.

A hypergraph $H'$ is a \emph{weak subhypergraph} of $H$ (called \emph{subhypergraph} by Bahmanian and Sajna) if $V'\subseteq V$ and $E'\subseteq\{e\cap V': e\in E\}$; that is, each hyperedge of $H'$ can be extended to one of $H$ by adding vertices of $V\setminus V'$ to it.  
For $V'\subseteq V$, the \emph{weak subhypergraph induced by $V'$}, denoted $H_w[V']$, has vertex set $V'$ and hyperedge set $E'=\{e\cap V': e\in E\}$.  Note that an induced weak subhypergraph might contain repeated edges and/or the empty edge.  To simplify our analysis, we tacitly replace $E'$ by $E'\setminus\{\emptyset\}$ and assume that weak subhypergraphs do not have multiple hyperedges (that is, $E'$ is a set, not a multiset).

Note that when $G$ is a (2-uniform) graph, strong subhypergraphs are the usual notion of subgraph, and weak subhypergraphs are subgraphs together with possible hyperedges of size one.  Each strong subhypergraph is also a weak subhypergraph but the reverse is not true.

Given hypergraphs $H_1$ and $H_2$, a weak (resp.\ strong) \emph{copy} of $H_1$ in $H_2$ is a weak (resp.\ strong) subhypergraph of $H_2$ isomorphic to $H_1$. 
Most of this paper is concerned with determining the existence of strong or weak copies of a fixed $H$ in $\sH(n,\ppp)$. In a mild abuse of terminology, we will often say that a hypergraph contains $H$ as a weak (strong) subhypergraph when we actually mean that the hypergraph contains a weak (strong) copy of $H$. The precise meaning will always be clear from the context.

\begin{figure}[ht]
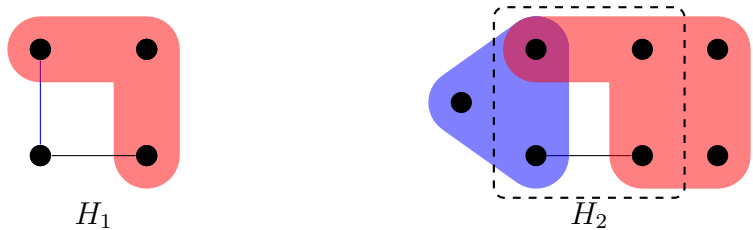

\begin{center}
\begin{hypergraph}
\foreach \i in {1,2,3,4}
{
  \node[vertex] (v\i) at (45+90*\i:1) {};
}
\foreach \i in {1,2,3,4}
{
  \node[vertex] at (45+90*\i:1) {};
}
\draw (v2)--(v3);
\begin{pgfonlayer}{background}
  \draw[blue] (v1)--(v2);
  \draw[edge, red] (v1.center)--(v4.center)--(v3.center);
\end{pgfonlayer}
\node at (0,-1.5) {$H_1$};
\end{hypergraph}
\hspace{7em}
\begin{hypergraph}
\foreach \i in {1,2,3,4}
{
  \node[vertex] at (45+90*\i:1) {};
}
\node[vertex] (u) at (-1.7,0) {};
\node[vertex] (w1) at ($(v3)+(1,0)$) {}; \node[vertex] (w2) at ($(v4)+(1,0)$) {};
\begin{pgfonlayer}{background}
  \draw (v2)--(v3);
  \draw[edge, blue] (v1.center)--(v2.center)--(u.center)--(v1.center);
  \draw[edge, red] (v1.center)--(v4.center)--(w2.center)--(w1.center)--(v3.center)--(v4)--(w1);
\end{pgfonlayer}
\node[draw, thick, dashed, rounded corners, inner sep=1em, fit=(v1) (v2) (v3) (v4)] {};
\node at (0,-1.5) {$H_2$};
\end{hypergraph}
\end{center}
\caption{The hypergraph $H_1$ appears as a weak subhypergraph of $H_2$ (induced by the dashed vertex subset), but not as a strong subhypergraph.\label{fig:weaksubgraph}}
\end{figure}

Since it should not cause confusion in this paper, we will often drop the affix ``hyper'': we refer to hyperedges as just edges and strong/weak subhypergraphs as just strong/weak subgraphs.  However, we will not drop the affix from ``hypergraph.''

\section{Small subgraphs in $\Hh(n,\ppp)$.}

We are interested in answering questions about the existence of subgraphs in $\Hh(n,\ppp)$. This question was addressed for $\G(n,p)$ by Bollob\'as in~\cite{Bollobas}. To state his result we require two definitions.
Let $G$ be a graph. Denote by $d(G)=|E(G)|/|V(G)|$ the \emph{density} of $G$, and by
$$
m(G) = \max \{d(G') : G' \subseteq G\}
$$
the \emph{maximum subgraph density} of $G$.
\begin{theorem} [Bollob\'as~\cite{Bollobas}] \label{thm:threshold_G}
For an arbitrary fixed graph $G$ with at least one vertex,
\begin{equation*}
\lim_{n\to\infty} \Prob \big(G \subseteq \G(n,p)\big) =
\begin{cases}
0& {\rm \ if\ }  np^{m(G)} \to 0 \\
1& {\rm \ if\ }  np^{m(G)} \to \infty.
\end{cases}
\end{equation*}
\end{theorem}
In other words, if $np^{m(G)} \to 0$, then a.a.s.\ $\G(n,p)$ does not contain $G$ as a subgraph. If $np^{m(G)} \to \infty$, then a.a.s.\ $\G(n,p)$ contains $G$ as a subgraph. The function $p_* = n^{-1/m(G)}$ (or any other function of the same asymptotic order) is called a \emph{threshold} probability for the property that $\G(n,p)$ contains $G$ as a subgraph.

Before we move to our result for random hypergraphs, let us mention why the maximum subgraph density of $G$, rather than simply the density of $G$, plays a role here. Consider the graphs $G$ and $G'$ depicted on Figure~\ref{fig:density}. Note that $G' \subset G$. 

\begin{figure}[h]
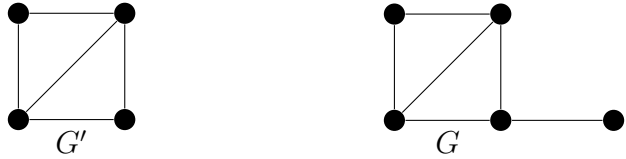

\begin{center}
\begin{hypergraph}
\foreach \i in {1,2,3,4}
{
  \node[vertex] (v\i) at (45+90*\i:1) {};
}
\draw (v1)--(v2)--(v3)--(v4)--(v1);
\draw (v2)--(v4);
\node at (0,-1) {$G'$};
\end{hypergraph}
\hspace{7em}
\begin{hypergraph}
\foreach \i in {1,2,3,4}
{
  \node[vertex] (v\i) at (45+90*\i:1) {};
}
\node[vertex] (v0) at ($(v3)+(1.5,0)$) {};  \draw (v0)--(v3);
\draw (v1)--(v2)--(v3)--(v4)--(v1);
\draw (v2)--(v4);
\node at (0,-1) {$G$};
\end{hypergraph}
\caption{\emph{Maximum subgraph density vs.\ density:} $G$ has density $6/5$, whilst its maximum subgraph density is $5/4$ ($>6/5$).  $G'\subset G$ has density $5/4$.}\label{fig:density}
\end{center}
\end{figure}

\noindent Take any function $p=p(n)$ such that $n^{-5/6} \ll p \ll n^{-4/5}$, say $p(n)=n^{-9/11}$. Then
\begin{eqnarray*}
  \E [X_n(G)] &=& \Theta(n^5 p^6) \to \infty \\
    \E [X_n(G')] &=& \Theta(n^4 p^5) \to 0,
\end{eqnarray*}
where $X_n(G)$ and $X_n(G')$ are random variables representing the number of copies of $G$ and $G'$ in $\G(n,p)$.  Since $\E [X_n(G)] \to \infty$, one might expect many copies of $G$; however, using the first moment method we get that a.a.s.\ there is no copy of $G'$ in $\G (n,p)$, and therefore a.a.s.\ there is no copy of $G$ either.

\bigskip

We now generalize Theorem~\ref{thm:threshold_G} to hypergraphs. 
In order to state our result, we need a few more definitions. Let $H = (V,E)$ be a hypergraph. Denote the number of vertices in $H$ by $v(H) = |V|$ and denote the number of edges by $e(H)=|E|$. For any $r \ge 1$, we will use $e_r(H) = | \{ e \in E : |e| = r\} |$ to denote the number of edges of size $r$ in $H$.  Finally, define
\begin{equation}\label{eq:mus}
\mu_s(H) = n^{v(H)} \prod_{r \ge 1} p_r^{e_r(H)}.
\end{equation}
We are now ready to state our result on the appearance of strong subgraphs of $\Hh(n,\ppp)$.  We adopt the convention that $0^0=1$ and assume all our hypergraphs have nonempty vertex set.

\begin{theorem}\label{thm:hypersubgraphs}
Let $H$ be an arbitrary fixed hypergraph. Let $\ppp = (p_r)_{r \ge 1}$ be any sequence
such that $0 \le p_r = p_r(n) \le 1$ for each $r \ge 1$. Let $\cJ$ denote the family of all strong subgraphs of $H$.
\begin{itemize}
\item [(a)] If for some $H' \in \cJ$ we have $\mu_s(H') \to 0$ (as $n \to \infty$), then a.a.s.\ $\Hh(n,\ppp)$ does not contain $H$ as a strong subgraph.
\item [(b)] If for all $H' \in \cJ$  we have $\mu_s(H') \to \infty$ (as $n \to \infty$), then a.a.s.\ $\Hh(n,\ppp)$ contains $H$ as a strong subgraph.
\end{itemize}
Moreover, if there exists $\eps >0$ such that $p_r \le 1-\eps$, for all $r$, then the conditions above determine whether or not $H$ appears as an \emph{induced} strong subgraph.
\end{theorem}

Let us mention that the result also holds in the multiset setting (i.e. when vertices are allowed to be repeated in each hyperedge with some multiplicity). This can be easily seen by replacing $\binom{n}{v(H)}v(H)!$ by $n^{v(H)}$ in the proof below and making several other trivial adjustments. 

\begin{proof}
Denote by $X_n=X_n(H)$ the random variable that counts strong copies of $H$ in a random hypergraph $\Hh(n,\ppp)$. Denote by $H_1, H_2, \dots, H_t$ all strong copies of $H$ in the complete hypergraph on $n$ vertices. Note that
$$
t = t(H) = \binom{n}{v(H)} \frac{v(H)!}{\aut(H)} = \Theta \left( n^{v(H)} \right) \,,
$$
where $\aut(H)$ is the number of automorphisms of $H$. For $i \in [t]$, let
\begin{equation*}
I_i = 
\begin{cases}
1& \text{if $H_i$ is a strong subgraph of $\Hh(n,\ppp)$}\\
0& \text{otherwise,}
\end{cases}
\end{equation*}
be an indicator random variable for the event that $H_i$ is a strong subgraph of $\Hh(n,\ppp)$.
Then $X_n = \sum_{i=1}^t I_i$. 

We start with part (a) of the statement. Let $H' \in \cJ$ with $\mu_s(H')\to0$. It follows from Markov's inequality that
\begin{eqnarray*}
\Prob(X_n(H) > 0) &\le& \Prob(X_n(H')>0) ~~\le~~ \E [X_n(H')] \\
&=& t(H') \prod_{r \ge 1} p_r^{e_r(H')} ~~=~~ \Theta \big( \mu_s(H')
\big) = o(1).
\end{eqnarray*}
Hence a.a.s.\ $X_n(H) = 0$ and part (a) is done.

For part (b), we need to estimate the variance:
\begin{eqnarray*}
\Var [X_n(H)] &=& \Var \left[\sum_{i=1}^t I_i \right] = \sum_{1 \le i,j \le t} \Cov (I_i, I_j) 
= \sum_{1 \le i,j \le t} \big(\E [I_i I_j] - (\E [I_i])(\E [I_j])\big)\\
&=& \sum_{1 \le i,j \le t} \Big(\Prob (I_i=1, I_j=1) - (\Prob (I_i=1))^2\Big)\\
&=& \sum_{1 \le i,j \le t} \Big(\Prob (I_i=1, I_j=1) - \prod_{r \ge 1} p_r^{2e_r(H)}\Big) \,.
\end{eqnarray*}
Observe that random variables $I_i$ and $I_j$ are independent if and only if $H_i$ and $H_j$ are edge-disjoint. In that case $\Prob (I_i=1, I_j=1) = \prod_{r \ge 1} p_r^{2e_r(H)}$, and such terms vanish from the above summation. 
Therefore, we may consider only the terms for which $e(H_i \cap H_j)\ge 1$.
For each $H' \subseteq_s  H$, there are $\Theta (n^{v(H')} n^{2(v(H)-v(H'))}) = \Theta (n^{2v(H)-v(H')})$ pairs $(H_i, H_j)$ of copies of $H$ in the complete hypergraph on $n$ vertices with $H_i \cap H_j$ isomorphic to $H'$. Thus,
\begin{eqnarray*}
\Var [X_n(H)] &=& \sum_{\substack{H' \subseteq_s H,\\ e(H')>0}} \Theta (n^{2v(H)-v(H')}) \left( \prod_{r\ge 1} p_r^{2e_r(H)-e_r(H')} - \prod_{r\ge 1} p_r^{2e_r(H)} \right) \\
&=& \sum_{\substack{H' \subseteq_s H,\\e(H')>0}} O \left( n^{2v(H)-v(H')} \prod_{r\ge 1} p_r^{2e_r(H)-e_r(H')} \right) \,.
\end{eqnarray*}
Since $\E [X_n(H)] = \Theta(n^{v(H)} \prod_{r \ge 1} p_r^{e_r(H)})$, we can use the second moment method to get
\begin{eqnarray*}
\Prob (X_n(H) = 0) &\le& \frac {\Var [X_n(H)]}{(\E [X_n(H)])^2} = \sum_{\substack{H' \subseteq_s H,\\e(H')>0}} O\left(n^{-v(H')} \prod_{r\ge 1} p_r^{-e_r(H')}\right) = o(1).
\end{eqnarray*}
Note that there are a finite number of terms in the above sum and, by assumption, each term tends to zero as $n\to \infty$. Hence a.a.s.\ $X_n(H) \ge 1$ and part (b) is done.
\end{proof}

In view of Theorem~\ref{thm:hypersubgraphs}, we emphasize that the existence of strong copies of $H$ in $\Hh(n,\ppp)$ cannot be determined by translating to graphs via the 2-section. For instance, consider the three hypergraphs $H_1$, $H_2$, and $H_3$ in Figure~\ref{fig:strongvs2sec}. Each of these has $H_1$ as its 2-section. However, the expected  number of strong copies of $H_1$, $H_2$ and $H_3$ in $\Hh(n,\ppp)$ is $n^4 p_2^5$, $n^4 p_2^2 p_3$, and $n^4 p_3^2$, respectively.  So if, say $p_3 = n^{-5/2}$ and $p_2 = n^{-3/4}$, then we expect many copies of $H_1$, a constant number of copies of $H_2$, and $o(1)$ copies of $H_3$. Moreover, by testing the conditions of Theorem~\ref{thm:hypersubgraphs} for all the strong subgraphs of $H_1,H_2$, and $H_3$, we obtain that
a.a.s.\ $\Hh(n,\ppp)$ contains $H_1$ as a strong subgraph, but not $H_3$ (and the theorem is inconclusive for $H_2$).
\begin{figure}[ht]
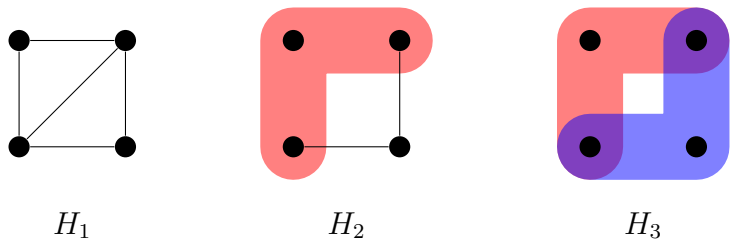

\begin{hypergraph}
\foreach \i in {1,2,3,4}
{
  \node[vertex] (v\i) at (45+90*\i:1) {};
}
\draw (v1)--(v2);
\draw (v2)--(v3);
\draw (v3)--(v4);
\draw (v4)--(v1);
\draw (v2)--(v4);
\node at (0,-1.75) {$H_1$};
\end{hypergraph}
\hspace{3em}
\begin{hypergraph}
\foreach \i in {1,2,3,4}
{
  \node[vertex] (v\i) at (45+90*\i:1) {};
}
\draw (v2)--(v3);
\draw (v3)--(v4);
\begin{pgfonlayer}{background}
\draw[edge,red] (v4.center)--(v1.center)--(v2.center);
\end{pgfonlayer}
\node at (0,-1.75) {$H_2$};
\end{hypergraph}
\hspace{3em}
\begin{hypergraph}
\foreach \i in {1,2,3,4}
{
  \node[vertex] (v\i) at (45+90*\i:1) {};
}
\begin{pgfonlayer}{background}
\draw[edge,red] (v4.center)--(v1.center)--(v2.center);
\draw[edge,blue] (v2.center)--(v3.center)--(v4.center);
\end{pgfonlayer}
\node at (0,-1.75) {$H_3$};
\end{hypergraph}
\caption{These three hypergraphs have the same 2-section, which is precisely $H_1$, but their behaviour as potential strong subgraphs of
$\Hh(n,\ppp)$ is different.\label{fig:strongvs2sec}}
\end{figure}

\bigskip

Next we consider the appearance of weak subgraphs of $\Hh(n,\ppp)$. 
For technical reasons, we restrict ourselves to hypergraphs with bounded edge sizes. 
Formally, for a given $M \in \N$, we say that $H=(V,E)$ is an \emph{$M$-bounded hypergraph} if $|e| \le M$ for all $e \in E$. 
Similarly, $\ppp = (p_r)_{r \ge 1}$ is an \emph{$M$-bounded sequence} if $p_r = 0$ for $r > M$. 
We will use $\ppp = (p_r)_{r = 1}^M$ for an $M$-bounded sequence instead of an infinite sequence $\ppp = (p_r)_{r \ge 1}$ with a bounded number of non-zero values. 
Clearly, if $\ppp$ is $M$-bounded, then so is $\Hh(n,\ppp)$ (with probability 1). 
For $r\in[M]$, let
\begin{equation}\label{eq:pprime}
p'_r = p_r + n p_{r+1} + n^2 p_{r+2} + \dotsb + n^{M-r} p_{M},
\end{equation}
and, given any fixed hypergraph $H$, define
\begin{equation}\label{eq:muw}
\mu_w(H) = n^{v(H)} \prod_{r = 1}^M (p_r')^{e_r(H)},
\end{equation}
which will play a role analogous to that of $\mu_s(H)$.

\begin{theorem}\label{thm:subhypergraphs}
Let $H$ be an arbitrary fixed hypergraph, and let $\cJ$ be the family of all strong subgraphs of $H$.  Let $\ppp = (p_r)_{r = 1}^M$ be an $M$-bounded sequence.
\begin{itemize}
\item [(a)]
If for some $H'\in\cJ$ we have $\mu_w(H') \to 0$ (as $n \to \infty$), then a.a.s.\ $\Hh(n,\ppp)$ does not contain $H$ as a weak subgraph.
\item [(b)] If for all $H'\in\cJ$ we have $\mu_w(H') \to \infty$ (as $n \to \infty$), then a.a.s.\ $\Hh(n,\ppp)$ contains $H$ as a weak subgraph.
\end{itemize}
\end{theorem}
\begin{proof}
Let $H' \in\cJ$ be a strong subgraph of $H$ for which $n^{v(H')} \prod_{r = 1}^M (p_r')^{e_r(H')} \to 0$. As before, if there is more than one strong subgraph with this property, then choose one arbitrarily. Note that the product above still tends to $0$ if all isolated vertices from $H'$ are removed, so we can assume that $H'$ has no isolated vertices. We will show that a.a.s.\ $\Hh(n,\ppp)$ does not contain $H'$ as a weak subgraph and so a.a.s.\ it does not contain $H$ as a weak subgraph either.  Let $\cF$ be the family of all $M$-bounded hypergraphs (up to isomorphism) with precisely $|E(H')|$ edges, no isolated vertices, and containing $H'$ as a weak subgraph. An important property is that each member of $\cF$ has a bounded number of vertices (trivially $M|E(H')|$) and so $\cF$ also has bounded size. Clearly, if $\Hh(n,\ppp)$ contains $H'$ as a weak subgraph, then $\Hh(n,\ppp)$ contains some member of $\cF$ as a strong subgraph.

Let us focus on any $F \in \cF$. For any $e \in E(H')$, let $\bar{e} \in E(F)$ be the corresponding edge in $F$ that $e$ is obtained from; that is, $e = \bar{e} \cap V(H')$. Observe that
\begin{eqnarray*}
\mu_s(F) =  n^{v(F)} \prod_{r = 1}^M p_r^{e_r(F)} &=& n^{v(F)} \prod_{\bar{e} \in E(F)} p_{|\bar{e}|} ~~=~~ n^{v(F)} \prod_{e \in E(H')} p_{|\bar{e}|}\\
&=& n^{v(F) - \sum_{e \in E(H')} (|\bar{e}| - |e|)} \prod_{e \in E(H')} n^{|\bar{e}| - |e|} p_{|\bar{e}|}\\
&\le& n^{v(H')} \prod_{e \in E(H')} p'_{|e|} ~~=~~ \mu_w(H') 
\to 0.
\end{eqnarray*}
It follows from Theorem~\ref{thm:hypersubgraphs}(a) that a.a.s.\ $\Hh(n,\ppp)$ does not contain $F$ as a strong subgraph.  As $|\cF|$ is bounded, a.a.s.\ $\Hh(n,\ppp)$ does not contain any member of $\cF$ as a strong subgraph and so a.a.s.\ $\Hh(n,\ppp)$ does not contain $H$ as a weak subgraph. Part (a) is finished. 

\smallskip

Let us move to part (b). It follows immediately from the definition of $p'_r$ that for each $r$ there exists $i(r) \in \Z$, $0 \le i(r) \le M-r$, such that $n^{i(r)} p_{r+i(r)} \ge p'_r / (M-r+1) \ge p'_r / M$. We construct a new hypergraph $J$ from $H$ as follows: for each $e \in E(H)$, add $i(|e|)$ new vertices to $V(J)$ and add them to $e$ to form $\bar{e} \in E(J)$.  See Figure~\ref{fig:g+} for an example of this construction. Our goal is to show that a.a.s.\ $\Hh(n,\ppp)$ contains $J$ as a strong subgraph, which will finish the proof as it implies that a.a.s.\ $\Hh(n,\ppp)$ contains $H$ as a weak subgraph.

\begin{figure}[htb]
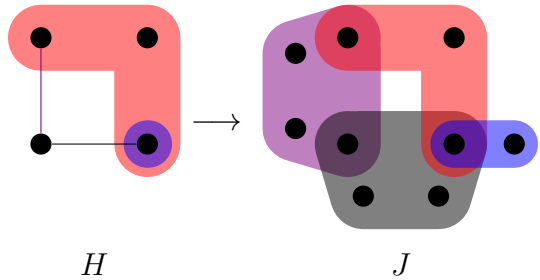

\begin{hypergraph}
\foreach \i in {1,2,3,4}
{
  \node[vertex] (v\i) at (45+90*\i:1) {};
}
\draw[violet] (v1)--(v2);
\draw (v2)--(v3);
\begin{pgfonlayer}{background}
\draw[edge, red] (v1.center)--(v4.center)--(v3.center);
\draw[edge, blue, line width=18pt] (v3.center)--(v3.center);
\end{pgfonlayer}
\node at (0,-2.3) {$H$};
\end{hypergraph}
\raisebox{5em}{$\longrightarrow$}
\begin{hypergraph}
\foreach \i in {1,2,3,4}
{
  \node[vertex] (v\i) at (45+90*\i:1) {};
}
\node[vertex] (u1) at (-1.4,0.5) {}; \node[vertex] (u2) at (-1.4,-0.5) {};
\node[vertex] (y1) at (-0.5,-1.4) {}; \node[vertex] (y2) at (0.5,-1.4) {};
\node[vertex] (w) at ($(v3)+(0.8,0)$) {};
\begin{pgfonlayer}{background}
\draw[edge, violet] (v1)--(v2.center)--(u2.center)--(u1.center)--(v1.center);
\draw[edge, black] (v2)--(v3.center)--(y2.center)--(y1.center)--(v2.center);
\draw[edge, red] (v1.center)--(v4.center)--(v3.center);
\draw[edge, blue, line width=18pt] (v3.center)--(w.center);
\end{pgfonlayer}
\node at (0,-2.3) {$J$};
\end{hypergraph}
\caption{The construction of $J$ from $H$; here $i(1)=1$, $i(2)=2$, $i(3)=0$.
\label{fig:g+}}
\end{figure}

Let $J'$ be any strong subgraph of $J$, and let $H'$ be obtained from $J'$ as follows: $V(H') = V(J') \cap V(H)$ and $E(H') = \{ \bar{e} \cap V(H) : \bar{e} \in E(J') \}$ (i.e., $H'$ is the weak subgraph of $J'$ induced by $V(H)$). Note that $E(H') \subseteq E(H)$ and so $H'$ is a strong subgraph of $H$. This time, observe that
\begin{eqnarray*}
\mu_s(J') = n^{v(J')} \prod_{r = 1}^M p_r^{e_r(J')} &=& n^{v(J')} \prod_{\bar{e} \in E(J')} p_{|\bar{e}|} ~~=~~ n^{v(J')} \prod_{e \in E(H')} p_{|\bar{e}|}\\
&=& n^{v(J') - \sum_{e \in E(H')} i(|e|)} \prod_{e \in E(H')} n^{i(|e|)} p_{|\bar{e}|} ~~\ge~~ n^{v(H')} \prod_{e \in E(H')} \frac {p'_{|e|}}{M} \\
&=& \frac {n^{v(H')}}{M^{e(H')}} \prod_{r = 1}^M (p_r')^{e_r(H')} = \Omega \big(  \mu_w(H')  \big) \to \infty.
\end{eqnarray*}
It follows from Theorem~\ref{thm:hypersubgraphs}(b) that a.a.s.\ $\Hh(n,\ppp)$ contains $J$ as a strong subgraph. Part~(b) and the proof of the theorem is finished.
\end{proof}

We shall discuss a few details concerning Theorem~\ref{thm:subhypergraphs}.
First, it is possible that a.a.s.\ some graph occurs as a weak subgraph but not as a strong one. For example, if
\begin{equation}\label{eq:pexample}
p_1 = n^{-0.6},\quad
p_2 = n^{-0.9},\quad
p_3 = n^{-1.7},\quad
\text{and}\quad
p_4 = n^{-3.1},
\end{equation}
then a.a.s.\ $\Hh(n,\ppp)$ does not contain $H$ as a strong subgraph but a.a.s.\ it does contain $J$ (both presented in Figure~\ref{fig:g+}) and therefore a.a.s.\ it contains $H$ as a weak subgraph.  Next, observe that if we replace the collection of all strong subgraphs $\cJ$ in the statement of Theorem~\ref{thm:subhypergraphs} by the collection $\cJ_w$ of all \emph{weak} subgraphs of $H$, the theorem remains valid. This is trivially true for part~(b), since $\cJ_w\supseteq\cJ$. For part~(a), simply replace $H'\in\cJ$ by $H'\in\cJ_w$ in the proof, and note that the argument follows.  Finally, let us comment on the definition of $p'_r$, and introduce related parameters $p''_r$ and $p'''_r$, which will play a role later on. Our particular choice of $p'_r$ in~\eqref{eq:muw} and thus in the statement of Theorem~\ref{thm:subhypergraphs} is the simplest function from the equivalence class of all functions of the same order. However, it is arguably more natural to replace $p'_r$ with $p''_r$ which is asymptotically the expected number of edges to which a given set of size $r$ belongs. For $r\in[M]$, let
\begin{equation}\label{eq:p2prime}
p''_r = p_r + n p_{r+1} + {n \choose 2} p_{r+2} + \dotsb + {n \choose M-r} p_{M}.
\end{equation}
Note that $p'_r$ and $p''_r$ are of the same order. More precisely,
$$
(1+o(1)) \frac {p'_r}{(M-r)!} \le p''_r \le p'_r.
$$
Hence, $p'_r$ can be replaced in~\eqref{eq:muw} by the more natural (but less simple) $p''_r$, and Theorem~\ref{thm:subhypergraphs} remains valid. It is worth noting that both $p'_r$ and $p''_r$ can be greater than one or even tend to infinity as $n \to \infty$.  While $p''_r$ is not a probability, we can create a probabalistic version, $p'''_r$, that represents the probability that a set of size $r$ belongs to some edge:
\begin{equation}\label{eq:p3prime}
p'''_r = 1 - (1-p_r) (1-p_{r+1})^{n-r} (1-p_{r+2})^{n-r \choose 2} \dotsb (1-p_M)^{n-r \choose M-r}.
\end{equation}
Observe that, if $p'_r = o(1)$ (or equivalently $p''_r=o(1)$), then $p''_r,p''_{r+1},\dotsc,p''_M=o(1)$, and therefore
\begin{align}
p'''_r &= 1 - \exp \left( - (1+o(1)) \left( p_r + n p_{r+1} + {n \choose 2} p_{r+2} + \dotsb + {n \choose M-r} p_M \right) \right)
\notag\\
&= 1 - \exp \left( - (1+o(1)) p''_r \right) ~~\sim~~ p''_r,
\label{eq:p3primeb}
\end{align}
so $p''_r$ and  $p'''_r$ asymptotically coincide.

\section{Induced weak subgraphs}

Let us discuss how one can use Theorem~\ref{thm:subhypergraphs} to determine whether $H$ appears as an \emph{induced} weak subgraph of $\Hh(n,\ppp)$. This seems to be a more complex question than in the case of strong subgraphs. Trivially, if $H$ is a complete hypergraph on $k$ vertices, then every weak copy of $H$ in $\Hh(n,\ppp)$ is automatically also induced. Otherwise, the non-edges of $H$ play a crucial role in determining the existence of induced weak copies. Indeed, a weak subgraph $H$ of $\Hh(n,\ppp)$ is induced provided that, for every set $e$ of vertices of $H$ that do not form an edge, $e$ cannot be extended to an edge of $\Hh(n,\ppp)$ by adding vertices not in $H$.

First, we will give some conditions that forbid a.a.s.\ the existence of induced weak copies of $H$ in $\Hh(n,\ppp)$ (even if $H$ \emph{does} appear as a weak subgraph).

\begin{proposition}\label{prop:noweakinduced}
Let $H$ be an arbitrary fixed hypergraph on $k$ vertices with a non-edge of size $r$ ($1\le r\le k$). Suppose $p''_r \ge (k + \eps) \log n$ for some constant $\eps > 0$, then a.a.s.\
$H$ does not occur as an induced weak subgraph of $\Hh(n,\ppp)$.
\end{proposition}
\begin{proof}
If $\Hh(n,\ppp)$ contains a copy $H_1$ of $H$ as an induced weak subgraph, then there must be a set of $r$ vertices in $H_1$ that cannot be extended to an edge of $\Hh(n,\ppp)$ by only adding vertices outside of $H_1$. The expected number of such sets is
\[
{n \choose k}{k \choose r} (1-p_r) (1-p_{r+1})^{n-k} \dotsb (1-p_M)^{n-k \choose M-r} \le n^k \exp ( - (1+o(1)) p''_r) = o(1),
\]
so a.a.s.\ there are none.
\end{proof}

Proposition~\ref{prop:noweakinduced} implies that if $H$ is an induced weak subgraph of $\Hh(n,\ppp)$ of order $k$ and $p''_r \ge (k + \eps) \log n$, then $H$ must contain all possible edges of size $r$.  Let us return to our example hypergraph $H$ in Figure~\ref{fig:g+} and the probabilities in~\eqref{eq:pexample}.  Note that $p''_1 \sim {n \choose 2} p_3 \sim n^{0.3}/2$, thus, a.a.s.\ $H$ does not occur as an induced weak subgraph of $\Hh(n,\ppp)$, as not every vertex of $H$ belongs to an edge of size one.

On the other hand, suppose that $r\ge1$ is the size of the smallest non-edge of $H$ and assume that
\begin{equation}\label{eq:p3bound}
\max\{p'''_r, p'''_{r+1}, \dotsc,p'''_M\}\le1-\eps
\end{equation}
for some constant $\eps>0$. Then any given weak copy of $H$ in $\Hh(n,\ppp)$ is also induced with probability bounded away from zero. In this case, the calculations in the proof of Theorem~\ref{thm:subhypergraphs} are still valid -- with an extra $\Theta(1)$ factor -- and thus the conclusions of that theorem extend to induced weak subgraphs. Since verifying condition~\eqref{eq:p3bound} may sometimes be tedious, we give a simpler sufficient condition.

\begin{proposition}\label{prop:weakinduced}
Let $H$ be an arbitrary fixed hypergraph, and let $r$ be the size of its smallest non-edge. Suppose that $p_r \le 1-\eps$ for some constant $\eps > 0$ and that $p'_r = O(1)$ (and, as a result, $p''_r = O(1)$ too).
If the conditions in part~(b) of Theorem~\ref{thm:subhypergraphs} are satisfied, then
a.a.s.\ $\Hh(n,\ppp)$ contains $H$ as an induced weak subgraph.
\end{proposition}
\begin{proof}
Since $p'_{j+1} = (p'_j  - p_j) / n \le p'_j / n$, we inductively get that $p'_j\le 1-\eps$ for all $j$ with $r+1\le j\le M$. 
In particular, $p_j\le 1-\eps$ and $p''_j=O(1)$ for $r\le j\le M$. In view of this, for all such $j$,
\begin{align*}
1-p'''_j &= (1-p_j)  (1-p_{j+1})^{n-j} (1-p_{j+2})^{n-j \choose 2} \cdots (1-p_M)^{n-j \choose M-j} \\
&\ge \eps \exp \left( - (1+o(1)) \left( n p_{j+1} + {n \choose 2} p_{j+2} + \ldots + {n \choose M-j} p_M \right) \right) \\
&\ge \eps \exp \left( - (1+o(1)) p''_j \right) > \eps',
\end{align*}
for some constant $\eps'>0$, and thus~\eqref{eq:p3bound} holds. Hence the conclusion of Theorem~\ref{thm:subhypergraphs} extends to weak induced subgraphs. In particular, part~(b) of that theorem gives a sufficient condition for the a.a.s.\ existence of weak induced copies of $H$.
\end{proof}

Let us return to our example from Figure~\ref{fig:g+} and the probabilities given in~\eqref{eq:pexample} for the last time. Since $p''_2 \sim np_3 = n^{-0.7} = o(1)$, if the ``missing'' edges of size 1 are added to $H$, then a.a.s.\ the resulting graph occurs as an induced weak subgraph of $\Hh(n,\ppp)$.

\section{The 2-section of $\Hh(n,\ppp)$}

We begin by considering the question of whether a given graph $G$ appears as a subgraph of the 2-section of $\Hh(n,\ppp)$.  Again we assume that $G$ has no isolated vertices.

Let us start with some general observations that apply to any host hypergraph $\Hh$, not necessarily $\Hh(n,\ppp)$. Observe that $G\subseteq[\Hh]_2$ if and only if there is a weak subgraph $H$ of $\Hh$ such that $G$ is a spanning subgraph of $[H]_2$. So we may test for $G\subseteq[\Hh]_2$ by finding every hypergraph $H$ with $G$ a spanning subgraph of $[H]_2$ and applying Theorem~\ref{thm:subhypergraphs} to each. We can reduce the number of hypergraphs that need to be tested: if $H_1$ is a weak subgraph of $H_2$ and $H_2$ is a weak subgraph of $\Hh$, then $H_1$ is also a weak subgraph of $\Hh$. Note also that a spanning weak subgraph is actually a strong subgraph; it suffices to check only the minimal hypergraphs $H$ (with respect to the (strong) subgraph relation) that have $G$ as a spanning subgraph of their 2-section.

In $\Hh(n,\ppp)$ one can reduce the number of hypergraphs $H$ to be tested even further. Given a hypergraph $H$, we construct a new hypergraph $H'$ on the same vertices and form hyperedges by taking a subset of each hyperedge of $H$.  Any strong subgraph of $H'$ is called a \emph{subedge system of $H$}.
Note that if $H_1$ is a subedge system of $H_2$ and $H_2$ is a weak subgraph of $H$, it is not necessarily true that $H_1$ is a weak subgraph of $H$, but it is true a.a.s.\ for $H=\Hh(n,\ppp)$.

\begin{proposition}\label{thm:subedgeofweak}
Let $H_1$ and $H_2$ be fixed hypergraphs with $H_1$ a spanning subedge system of $H_2$, and let $\ppp$ be $M$-bounded. Let $\cJ_1$ and $\cJ_2$ denote the set of all strong subgraphs of $H_1$ and $H_2$, respectively. If every $H'_2\in\cJ_2$ satisfies $\mu_w(H'_2)\to\infty$, then every $H'_1\in\cJ_1$ also satisfies $\mu_w(H'_1)\to\infty$.
\end{proposition}
\begin{proof}
Let $H'_1\in\cJ_1$. 
Since edges of $H'_1$ are subsets of edges of $H_2$, we can turn $H'_1$ into a strong subgraph $H'_2$ of $H_2$ by appropriately extending some of its edges. For each edge $e_1$ of $H'_1$ of size $r$ that is extended to an edge $e_2$ of size $r+i$, one factor of $p'_r$ in $\mu_w(H'_1)$ is replaced by one corresponding factor of $p'_{r+i}$ in $\mu_w(H'_2)$. We have that $p'_r\geq n^i p'_{r+i} \ge p'_{r+i}$, so $\mu_w(H'_1) \ge \mu_w(H'_2) \to \infty$.
\end{proof}

\begin{corollary}
Fix a graph $G$ without isolated vertices.  Let $\mathcal{F}$ denote the family of minimal (with respect to the subedge system relation) hypergraphs containing $G$ in their 2-section.  Let $\ppp$ be $M$-bounded.
\begin{enumerate}
\item[(a)] If for every $H\in\mathcal{F}$ there is some strong subgraph $H' \subseteq_s  H$ with $\mu_w(H') \to0$, then a.a.s.\ $G$ is not a subgraph of $[\Hh(n,\ppp)]_2$.
\item[(b)] If for some $H\in\mathcal{F}$ every strong subgraph $H' \subseteq_s  H$ satisfies $\mu_w(H') \to\infty$, then a.a.s.\ $G$ is a subgraph of $[\Hh(n,\ppp)]_2$.
\end{enumerate}
\end{corollary}

Next we consider the following question: suppose a copy of $G$ is found in $[\Hh(n,\ppp)]_2$. What is the probability that this copy comes from a given weak subgraph of $\Hh(n,\ppp)$?

Let $G$ be a fixed graph with no isolated vertices.  Let $\mathcal{F}$ denote the family of hypergraphs $H$ on the same vertex set as $G$ such that $G \simeq [H]_2$. Then, $G$ appears as an induced subgraph of $[\sH(n,\ppp)]_2$ if and only if some $H\in\cF$ appears as an induced weak subgraph of $\sH(n,\ppp)$. More precisely, for every set of vertices $S$ inducing a copy of $G$ in $[\sH(n,\ppp)]_2$, there is exactly one $H\in\cF$ such that $S$ induces a weak copy of $H$ in $\sH(n,\ppp)$. In this case we say that the hypergraph $H$ \emph{originates} that particular copy of $G$.  As a result we have the following corollary. 

\begin{proposition}\label{prop:origin}
Let $\ppp = (p_r)_{r = 1}^M$ be an $M$-bounded sequence. For $r\in[M]$, let $p'''_r$ be defined as in~\eqref{eq:p3prime}. Given a copy of $G$ in $[\sH(n,\ppp)]_2$, the probability that it originates from a given $H\in\cF$ is
\[
(1+o(1))\frac{\aut(H) \prod_{r=1}^M (p'''_r)^{e_r(H)}  (1 - p'''_r)^{{v(G) \choose r} - e_r(H)} }{\sum_{H'\in\cF} \aut(H') \prod_{r=1}^M (p'''_r)^{e_r(H')} (1 - p'''_r)^{{v(G) \choose r} - e_r(H')} }.
\]
\end{proposition}

Define the \emph{signature} of $H\in\cF$ as the vector  $\ev(H)=(e_1(H),e_2(H),\dotsc,e_{k}(H))$, 
where $k=v(G)$ (and hence also $k=v(H)$). 
Let $\ev(\cF) = \{\ev(H) : H\in \cF\}$. 
For a given signature 
$\ev\in\ev(\cF)$, let $\cF_\ev \subseteq \mathcal{F}$ be the family of hypergraphs in $\mathcal{F}$ with signature $\ev$.  Notice that
$\{\cF_\ev : \ev \in \ev(\cF) \}$ is a partition of $\mathcal{F}$. We can state the following corollary to Proposition~\ref{prop:origin} which we will make use of when comparing model predictions to real-world networks (see Section~\ref{sec:experiments}). 

\begin{corollary}\label{cor:origin}
Let $\ppp = (p_r)_{r = 1}^M$ be an $M$-bounded sequence. For $r\in[M]$, let $p'''_r$ be defined as in~\eqref{eq:p3prime}.
Then, given a copy of $G$ in $[\sH(n,\ppp)]_2$, the probability that it originates from a hypergraph with a given signature $\ev=(m_1,m_2,\ldots,m_k)\in\ev(\cF)$ is
\[
(1+o(1))\frac{ \sum_{H \in \cF_\ev} \aut(H) \prod_{r=1}^k (p'''_r)^{m_r}(1-p'''_r)^{\binom{v(G)}{r}-m_r}}
{\sum_{H'\in\cF} \aut(H') \prod_{r=1}^k (p'''_r)^{e_r(H')}(1-p'''_r)^{\binom{v(G)}{r}-e_r(H')}} .
\]
\end{corollary}

We are particularly interested in the subgraphs induced by the action of 2-sectioning, i.e. complete graphs. Of course, these include as subgraphs all sparser graphs on the same or fewer vertices.  Let us briefly explore the implications of Corollary~\ref{cor:origin} when $G$ is a complete graph.  Let $K_k$ denote the complete graph on $k$ vertices, $k\ge2$, and suppose that $\ppp$ is an $M$-bounded sequence satisfying $\binom{n}{j} p_j = O(n)$ for all $j\in[M]$. The latter condition is equivalent to assuming that the expected number of edges of each given size is at most linear in the number of vertices, which is a fairly reasonable assumption for many hypergraph networks. Additionally, suppose that for some $r$ with $k\le r\le M$ we also have $\binom{n}{r} p_r = \Omega(n)$. 
From~\eqref{eq:p3primeb}, we obtain that $p'''_j = O(1/n^{j-1})$ for every $j\in[M]$ and $p'''_k = \Theta(1/n^{k-1})$. 
Consider the signature $\widehat\ev=(0,\dotsc,0,1)$ corresponding to the hypergraph $\widehat H$ on $k$ vertices with a single edge of size $k$.
A straightforward inductive argument reveals that, for any signature $\ev=(m_1,m_2,\ldots,m_k)\in\ev(\cF)$,
\[
\prod_{r=1}^k (p'''_r)^{m_r}(1-p'''_r)^{\binom{k}{r}-m_r} = \begin{cases}
(1+o(1))p'''_k = \Theta(1/n^{k-1}) & \text{if $\ev=\widehat\ev$}
\\
o(1/n^{k-1}) & \text{if $\ev\ne\widehat\ev$.}
\end{cases}
\]
As a result, applying Corollary~\ref{cor:origin} to all signatures different from $\widehat\ev$, we conclude that, for a given copy of $K_k$ in $[\sH(n,\ppp)]_2$, a.a.s.\ it must originate from $\widehat H$.

\section{Comparing the model with reality}\label{sec:experiments}

In this section, we look at two real-world datasets that are naturally represented as hypergraph networks. We consider how well our model captures certain features of these datasets by looking at the appearance of select weak subhypergraphs and by comparing a measure of clustering coefficient.  

\subsection{Real-world datasets}\label{section:real-world-datasets}

We examine two real-world datasets; a coauthorship hypergraph and an email hypergraph.

\bigskip

The \textbf{email hypergraph} was constructed from the Enron dataset; a version of which can be obtained from Carnegie Mellon University~\cite{Enron_link}.  This dataset consists of 30,100 email messages from 151 Enron employees.  We use the messages sent by these individuals to build an undirected hypergraph for analysis.  From each message we extract the $from$, $to$, $cc$ and $bcc$ fields. The fields are merged (removing repeated addresses) and the resulting set is treated as an undirected hyperedge.  We recognize that this data might be better represented as directed hyperedges, but that is outside the scope of this paper and may be considered in future work.  For the purpose of this paper we also ignore the effects of multiple identical hyperedges, leaving us with 11,407 unique undirected hyperedges. The distributions of degrees and edge sizes can be seen in Figure~\ref{fig:Enron-stats}(a) and Figure~\ref{fig:Enron-stats}(b).  Note that the degree of a vertex is defined to be the number of hyperedges the vertex is contained in, while the edge size is defined to be the number of vertices in the hyperedge.  These two definitions are unambiguous in this case as edges do not contain repeated vertices. 

\begin{figure}[htbp]
\begin{center}
\begin{tabular}{cc}
\includegraphics[width=3in]{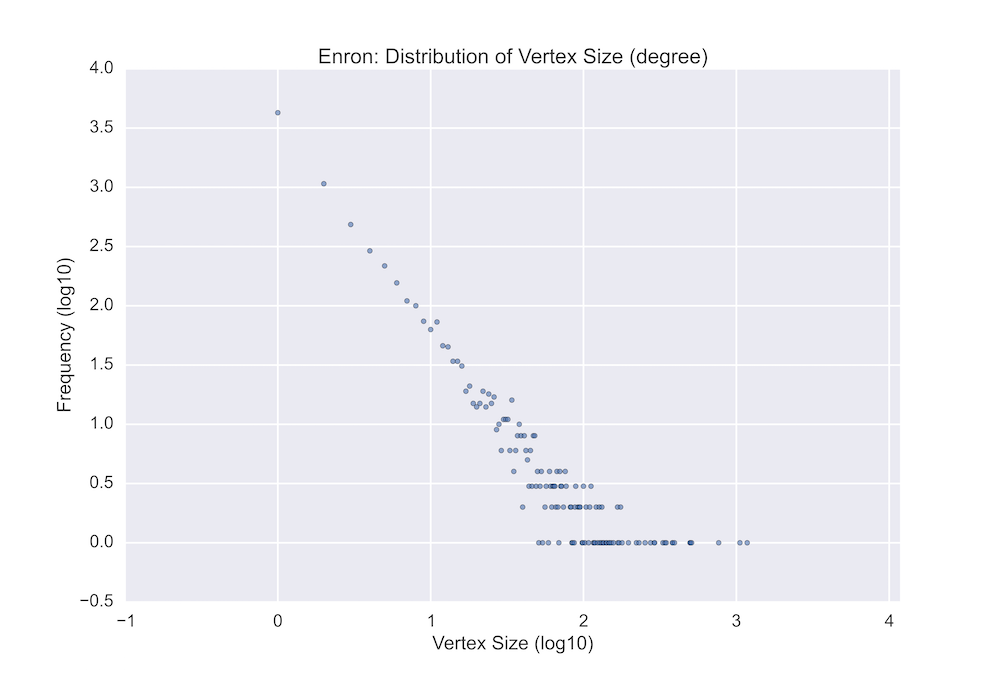} & \includegraphics[width=3in]{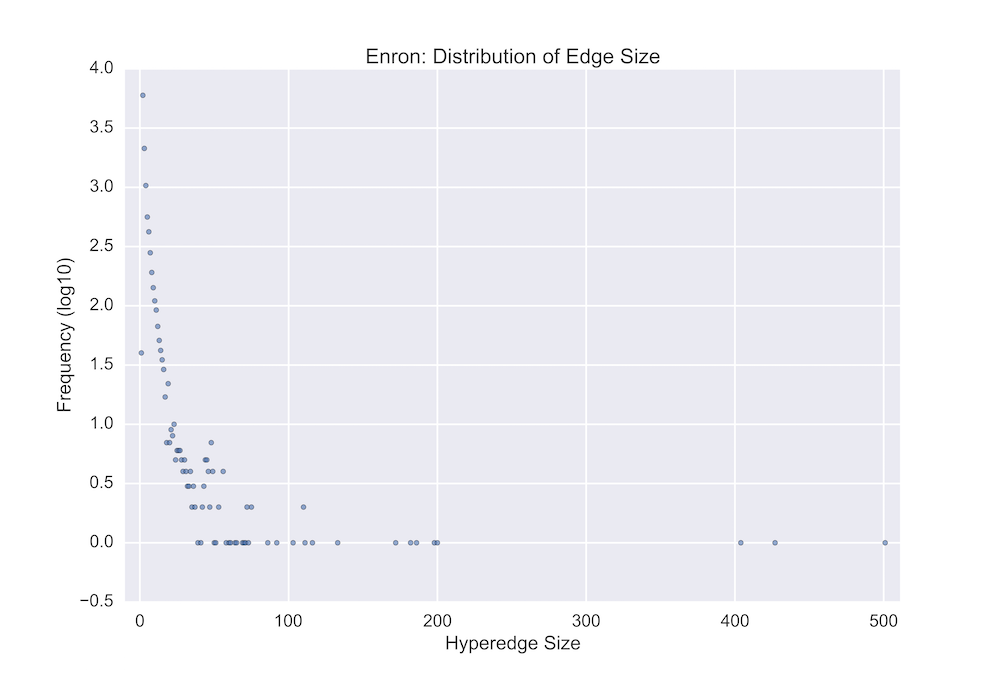} \\
(a) & (b)
\end{tabular}
\caption{Degree distribution (a) and edge-size distribution (b) of the email hypergraph from the Enron dataset.}
\label{fig:Enron-stats}
\end{center}
\end{figure}

The \textbf{coauthorship hypergraph} was generated from the ArnetMiner (AMiner) dataset.  ArnetMiner~\cite{Tang2008} is a project which attempts to extract researcher social networks from the World Wide Web.  Of particular interest to us is that it integrates a number of existing digital libraries.  The dataset built with this tool is available for research purposes at ~\cite{coauthorship_link}.  It consists of 2,092,356 research papers from 1,712,433 authors. From this we construct a coauthorship hypergraph with vertices representing authors and with one hyperedge for each paper in the database.  Hyperedges are undirected and consist of the set of authors listed on each paper. Again, we eliminate duplicate edges resulting in a hypergraph containing 1,499,404 unique hyperedges.  The distributions of degrees and edge sizes can be seen in Figure~\ref{fig:coauthor-stats}(a) and Figure~\ref{fig:coauthor-stats}(b).

\begin{figure}[htbp]
\begin{center}
\begin{tabular}{cc}
\includegraphics[width=3in]{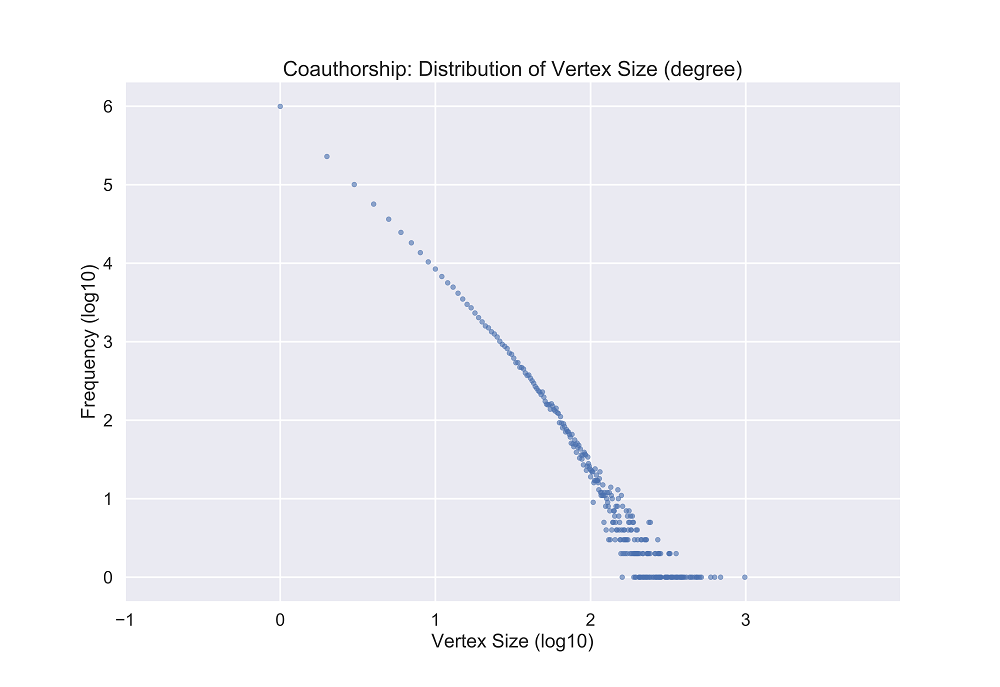} &
\includegraphics[width=3in]{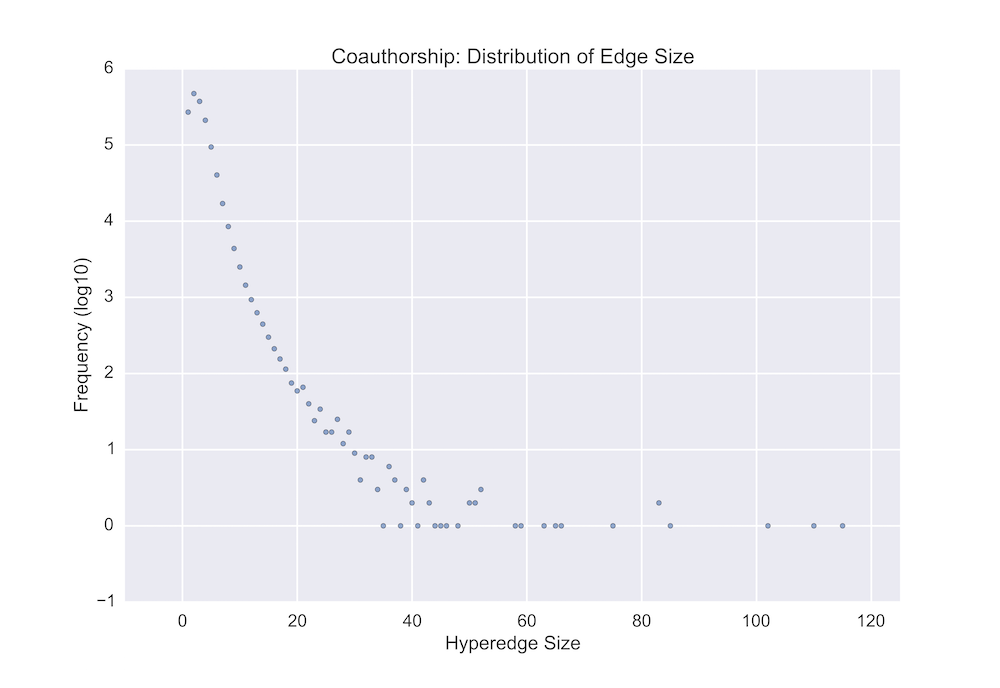} \\
(a) & (b) \\
\end{tabular}
\caption{Degree distribution (a) and edge-size distribution (b) of the coauthorship hypergraph from the AMiner dataset.}
\label{fig:coauthor-stats}
\end{center}
\end{figure}

In our analysis we will consider truncated hypergraphs of those described above.  The main reason is that large hyperedges have a significant effect on computation time and can drown out the signal of smaller, more interesting, effects.  In particular, large hyperedges can have a significant impact on the appearance of small complete graphs in the 2-section.  For example, each edge of size 100 introduces ${100 \choose k}$ copies of $K_k$. There are sporadic hyperedges of such size in both networks we investigate. 

Recall that $m_i$ denotes the number of hyperedges of size $i$. In the email hypergraph we have $m_2 = 5{,}975$, $m_3 = 2{,}128$, $m_4 = 1{,}034$, and $m_5 = 561$. After removing hyperedges of size greater than or equal to 6, we were left with a hypergraph, $\mathscr{E}_5$, on $n_5 = 5{,}044$ vertices and having $m = \sum_{i=2}^5 m_i = 9{,}698$ edges.  After removing hyperedges of size greater than or equal to 5, we were left with a hypergraph, $\mathscr{E}_4$, on $n_4 = 4{,}919$ vertices and having $m = \sum_{i=2}^4 m_i = 9{,}137$ edges. In the coauthorship hypergraph we have $m_2 = 473{,}560$, $m_3 = 373{,}262$, $m_4 = 211{,}011$, and $m_5 = 94{,}168$.  After removing hyperedges of size greater than or equal to 6, we were left with a hypergraph, $\mathscr{D}_5$, on $n_5 = 1{,}264{,}602$ vertices having $m = \sum_{i=2}^5 m_i = 1{},152{,}001$ edges. After removing hyperedges of size greater than or equal to 5, we were left with the hypergraph $\mathscr{D}_4$ on $n_4 = 1{,}146{,}130$ vertices having $m = \sum_{i=2}^4 m_i = 1{,}057{,}833$ edges. 

Note that all isolated vertices were also removed from the networks.

\subsection{Creating the model hypergraphs}

For each real-world hypergraph we analyse, we create model hypergraphs having the same expected edge counts.  That is, if $m_i$ is the number of edges of size $i$ in the hypergraph we wish to model, then we create a random hypergraph $\Hh(n,\ppp)$ on $n$ vertices where each $i$-set forms a hyperedge of size $i$ with probability $p_i = m_i / {n \choose i}$. Note that the model hypergraph created has (an expected number of) $m_i$ hyperedges of size $i$ randomly distributed throughout the graph. For our comparisons we generated 1048 model hypergraphs for each of the four real-world datasets.

In Figure~\ref{fig:real_v_model_degdist} we compare the degree distribution of the email hypergraph with edges of maximum size 5, $\mathscr{E}_5$ (a), and the coauthorship hypergraph with edges of maximum size 5, $\mathscr{D}_5$ (b), to their corresponding model hypergraphs. 

\begin{figure}[htbp]
\begin{center}
\begin{tabular}{cc}
\includegraphics[width=3in]{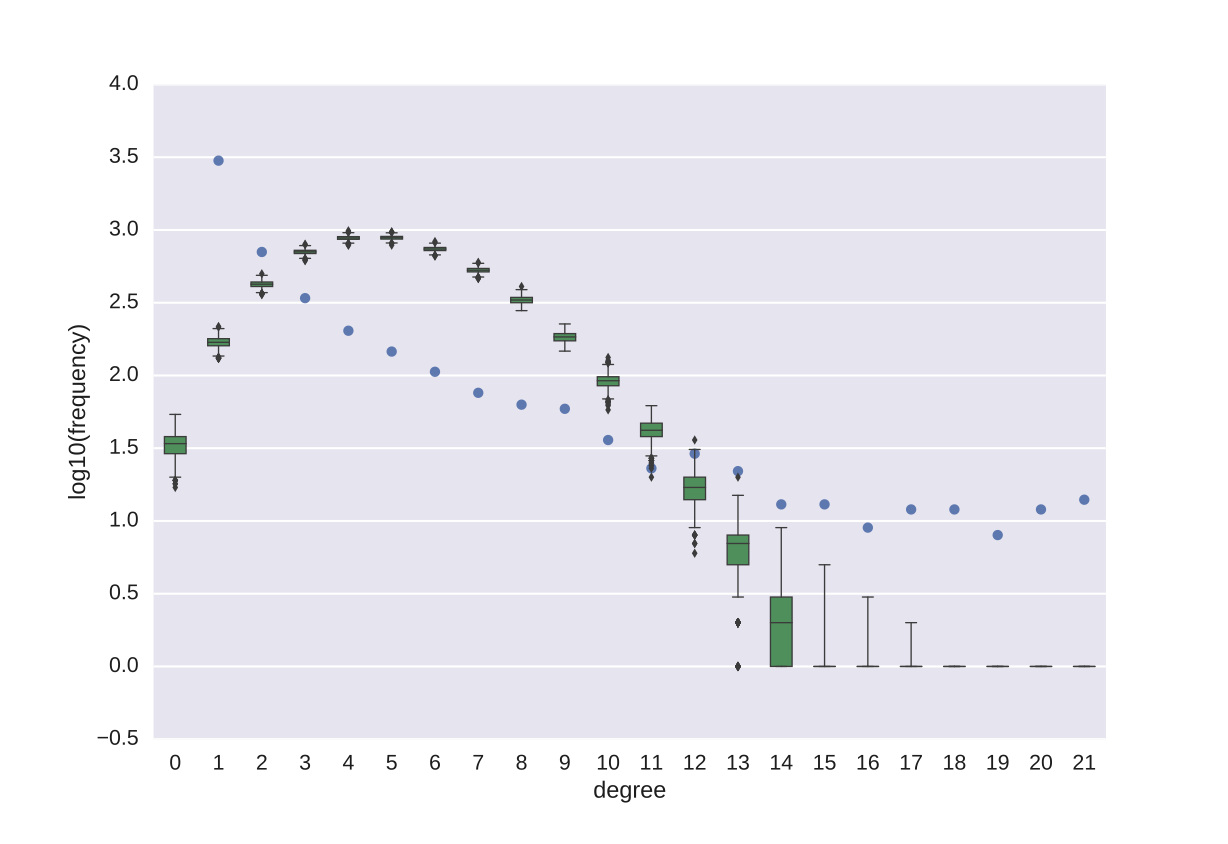} &
\includegraphics[width=3in]{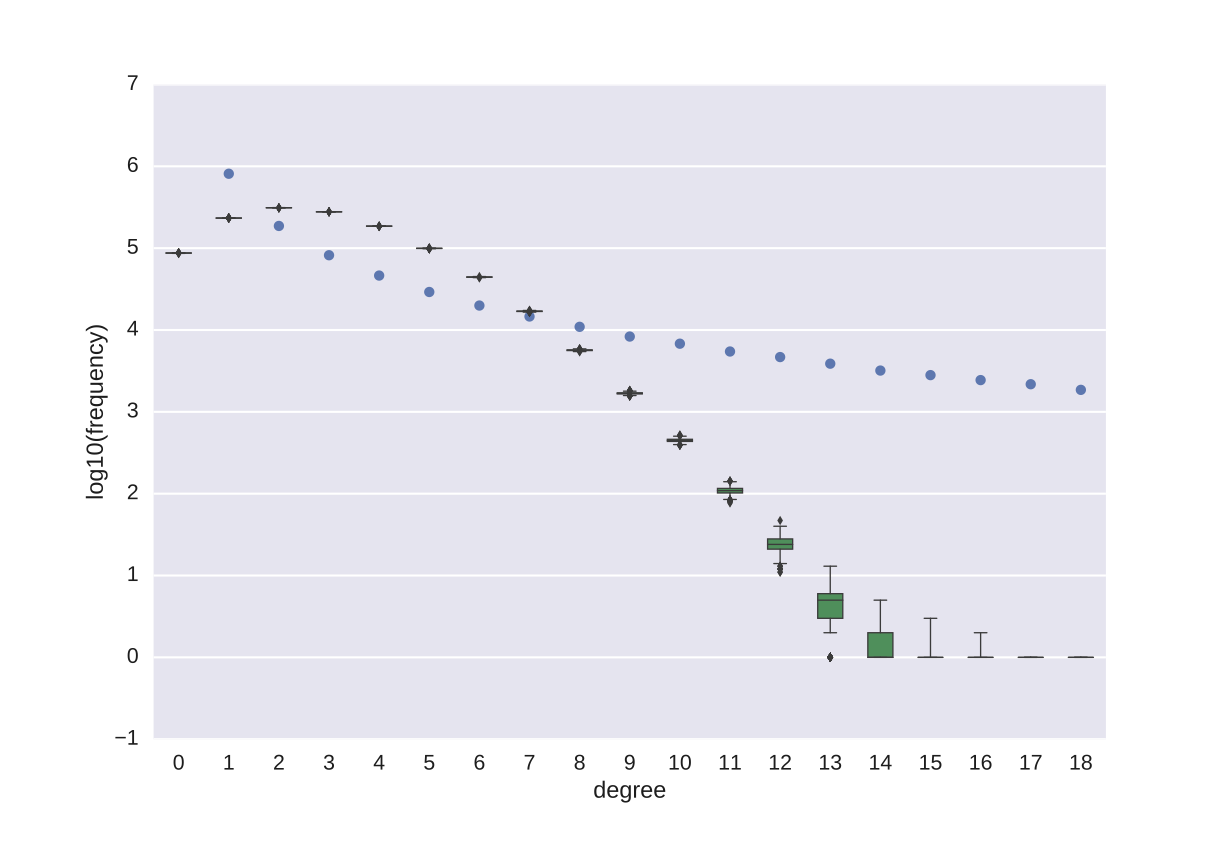} \\
(a) & (b) \\
\end{tabular}
\caption{Comparing degree distribution of datasets (blue dots) and generated model hypergraphs (boxplot to indicate variance): $\mathscr{E}_5$ (a) and $\mathscr{D}_5$ (b).}
\label{fig:real_v_model_degdist}
\end{center}
\end{figure}

While it is clear from these plots that the model produces hypergraphs significantly different from reality, sections~\ref{subsec:signatures} and~\ref{subsec:ccc} further investigate these differences.

\subsection{Comparing the appearance of signatures in theory and practice}\label{subsec:signatures}

Recall that the signature of a hypergraph $H$ is defined as $\ev(H)=(e_1(H),e_2(H),\dotsc,e_{k}(H))$, 
where $k=v(H)$.  Given a subset of vertices from a hypergraph, we will consider the non-$e(1)$ portion of the signature of the weak subhypergraph induced by these vertices. For our analysis, we fix a small graph $G$.  We look for this graph $G$ in the 2-section, $[H]_2$, of the network and then look at $H$ to determine the signature of the originating weak subhypergraph.

For practical purposes, we focus on the complete graphs $K_4$ and $K_5$.  Smaller complete graphs occur too frequently in the 2-section, whereas larger complete graphs do not occur frequently enough for us to claim anything statistically reliable. Of course, other subgraphs appear in $[H]_2$ but we are particularly interested in those specifically induced by the 2-section action.  We use Corollary~\ref{cor:origin} to estimate the probability that a copy of $K_4$ or $K_5$ taken  uniformly at random from the 2-section of $\Hh(n,\ppp)$ originates from a weak subhypergraph with a given signature in $\Hh(n,\ppp)$.

There are 70 signatures of hypergraphs on 4 vertices.  Of these there are 60 that have at least one hypergraph inducing $K_4$ in the 2-section; call these \emph{feasible}.  For example, there are $36 = {6 \choose 1} {4 \choose 2}$ labelled hypergraphs with one 2-edge and two 3-edges, but only 6 of these induce $K_4$ in the 2-section.  The number of signatures of hypergraphs on 5 vertices that can induce $K_5$ in the 2-section is 1,422. In order to apply Corollary~\ref{cor:origin} we generated all feasible signatures and all realizations of the subhypergraphs that induce a complete graph of the appropriate size in the 2-section~\cite{program}.

As the results are similar for all four real-world hypergraphs described above, we choose  $\mathscr{E}_5$ for illustration purposes. In Table~\ref{table:Enron_theory} we list the most popular signatures from the theoretical point of view. Note that theoretical estimations for probabilities are decreasing very fast as we go down. As a result, it makes more sense to compare how signatures rank in both models instead of comparing their corresponding probabilities.

We ranked the 179 (of a possible 1{,}422) observed signatures based on their number of occurrences in $\mathscr{E}_5$; we denote this $R_{observed}$.  We also ranked the 179 observed signatures based on their theoretical probabilities; we denote this $R_{theory/observed}$.  Some theoretically popular signatures do not occur at all in practice.  For example, signatures (0,1,0,1) and (2,1,1,0)  were the 4th and the 7th most popular signatures from theoretical predictions; we call these ranks $R_{theory}$. 
 
\begin{table}[htbp]
\caption{$\mathscr{E}_5$ network: popular signatures in theory}
{\tiny
\begin{center} 
\begin{tabular}{|c|c|c|c|c|c|}
\hline
Signature & Probability & $R_{theory}$ & $R_{theory/observed}$ & Probability & $R_{observed}$ \\
          & (theory) &       &  & (observed) &  \\
\hline
0,0,0,1	& 9.8118419955e-01	& 1	& 1 & 0.004911591	& 56\\
1,0,0,1	& 1.8649018608e-02	& 2	& 2 & 0.014734774	& 18\\
2,0,0,1	& 1.5950486447e-04	& 3	& 3 & 0.015717092	& 16\\
0,1,0,1	& 5.4464266334e-06	& 4	&  & - & \\
3,0,0,1	& 8.0844057265e-07	& 5	& 4 & 0.010805501	& 31\\
4,0,1,0	& 4.4141198260e-07	& 6	& 5 & 0.000982318	& 148\\
2,1,1,0	& 4.0695426832e-07	& 7	&  & - & \\
1,1,0,1	& 1.0351829114e-07	& 8	& 6 & 0.004911591	& 57\\
0,2,1,0	& 3.1265534058e-08	& 9	&  & - & \\
1,0,2,0	& 1.8314270051e-08	& 10 &  & - & \\
3,1,1,0	& 6.1878678650e-09	& 11 & 7 & 0.000982318	& 141 \\
5,0,1,0	& 5.0338562002e-09	& 12 & 8 & 0.002946955	& 89 \\
4,2,0,0	& 2.8645459166e-09	& 13 &  & - &  \\
4,0,0,1	& 2.6890048533e-09	& 14 & 9 & 0.013752456	& 21 \\
2,1,0,1	& 8.8539088013e-10	& 15 & 10 & 0.011787819	& 27 \\
\hline
\end{tabular}
\end{center}
}
\label{table:Enron_theory}
\end{table}%

\begin{table}[htbp]
\caption{$\mathscr{E}_5$ network: popular signatures in practice}
{\tiny
\begin{center} 
\begin{tabular}{|c|c|c|c|c|}
\hline
Signature & Probability & $R_{theory/observed}$ & Probability & $R_{observed}$ \\
          & (theory) &  & (observed) &  \\
\hline
4,1,0,1	& 1.4926318276e-14	& 24	& 0.033398821	& 1 \\
4,5,1,0	& 9.9459608645e-34	& 119	& 0.031434185	& 2 \\
4,4,1,0	& 1.1916310154e-27	& 84	& 0.026522593	& 3\\
3,2,0,1	& 1.1209412171e-17	& 36	& 0.02259332	& 4\\
4,2,0,1	& 3.7284328311e-20	& 48	& 0.02259332	& 5\\
4,4,2,0	& 4.7383972663e-37	& 133	& 0.02259332	& 6\\
9,1,0,0	& 2.5600700779e-16	& 31	& 0.02259332	& 7\\
10,1,0,0& 1.6219446713e-19	& 45	& 0.021611002	& 8\\
4,6,0,0	& 6.3964285475e-31	& 101	& 0.021611002	& 9\\
7,3,0,0	& 3.2971609825e-21	& 49	& 0.019646365	& 10\\
\hline
\end{tabular}
\end{center}
}
\label{table:Enron_obs}
\end{table}%

To understand these results at a glance, we took the 179 observed signatures and plotted the observed ranking ($R_{observed}$) against their theoretical ranking ($R_{theory/observed}$) in Figure~\ref{fig:Enron}.  Figure~\ref{fig:Enron} also contains a similar plot for the $\mathscr{E}_4$ dataset.  It is clear from these two plots that the theoretical and observed ranks are highly uncorrelated.

\begin{figure}[htbp]
\begin{center}
\begin{tabular}{cc}
\includegraphics[scale=0.5]{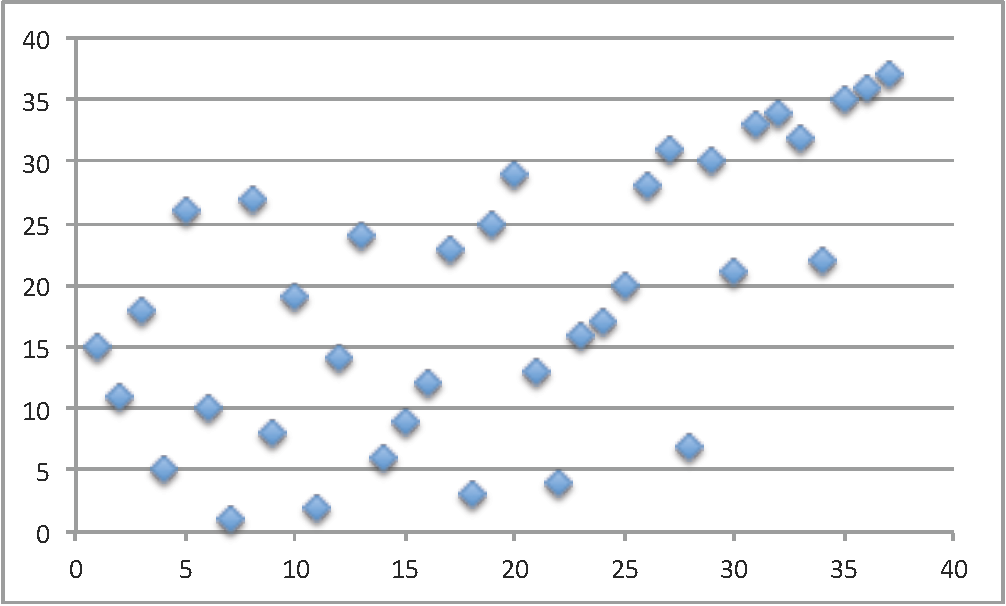} & \includegraphics[scale=0.5]{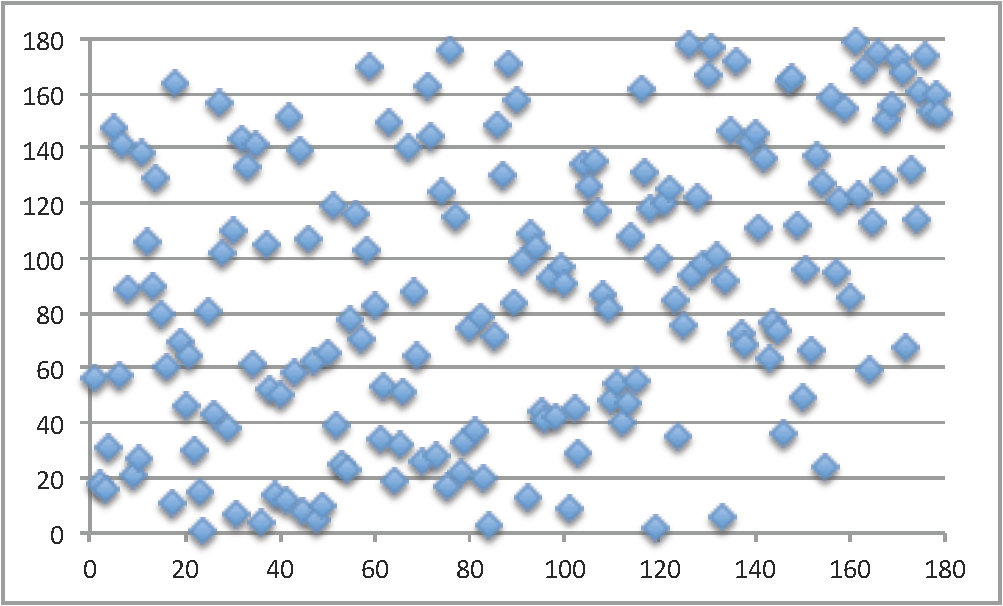} \\
(a) $\mathscr{E}_4$ network & (b) $\mathscr{E}_5$ network \\
\end{tabular}
\caption{Theoretical rank, $R_{theory/observed}$ vs.\ Observed rank, $R_{observed}$ ($x$-axis vs. $y$-axis)}
\label{fig:Enron}
\end{center}
\end{figure}

In Figure~\ref{default} we present the same comparison of the observed versus theoretical observed ranks of signatures for the coauthorship hypergraph.

\begin{figure}[htbp]
\begin{center}
\begin{tabular}{cc}
\includegraphics[scale=0.5]{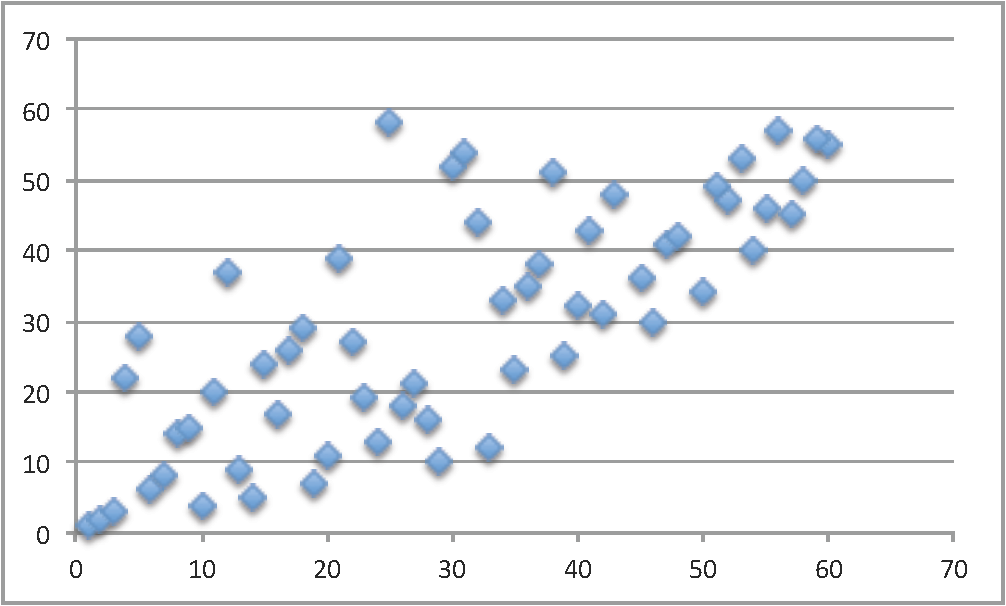}
& \includegraphics[scale=0.5]{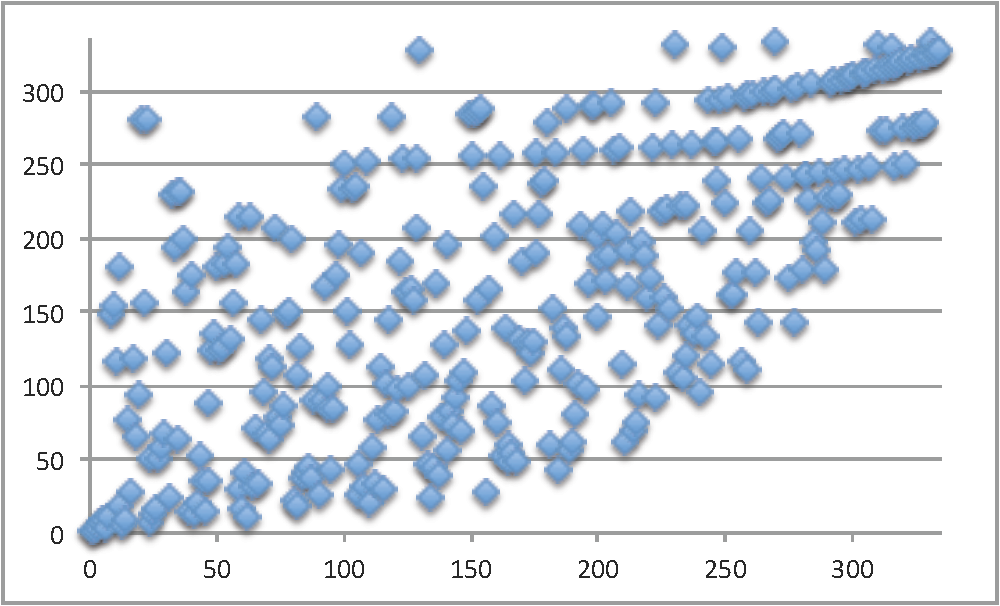} \\
(a) $\mathscr{D}_4$ & (b) $\mathscr{D}_5$ \\
\end{tabular}
\caption{Theoretical rank, $R_{theory/observed}$ vs. Observed rank, $R_{observed}$ ($x$-axis vs. $y$-axis)}
\label{default}
\end{center}
\end{figure}

These rankings confirm our expectation that the simple random hypergraph model we have introduced does not capture features exhibited in real-world datasets.  In particular, the model assumes edges occur independently, which is generally not the case in real-world networks and certainly not the case in the two datasets we considered. In order to better understand the interdependence of hyperedges in real-world networks, some notion of clustering coefficient must be investigated in the hypergraph setting. 

\subsection{Comparing clustering coefficient in theory and practice}\label{subsec:ccc}

Let us recall the definition of the clustering coefficient for graphs.  The clustering coefficient is an attempt to measure the degree to which vertices in a graph $G=(V,E)$ tend to cluster together by focusing on connections between neighbours of a vertex. The \emph{local clustering coefficient} $c(v)$ for a vertex $v$ is given by the proportion of links between the vertices within its (open) neighbourhood (i.e. not including $v$) divided by the number of links that could possibly exist between them; that is, for vertices of degree at least $2$
\begin{eqnarray*}
c(v) &=& \frac { |\{ xy \in E : x,y \in N(v) \}| }{ {|N(v)| \choose 2} } \\
       &=& \frac { \text{number of triangles involving } v }{\text{number of pairs of edges adjacent at v} },
\end{eqnarray*}
while for vertices of degree less than $2$, $c(v) =0$. The \emph{clustering coefficient} of a graph, with at least one vertex of degree at least $2$, can be defined as the average of the local clustering coefficients of all the vertices; that is, when $|V_{\deg \geq 2}|\ge 1$,
$$
C(G) = \frac {1}{|V_{\deg \geq 2}|} \sum_{v \in V} c(v),
$$
where $|V_{\deg \geq 2}|$ is the number of vertices of degree at least $2$. Note that this clustering coefficient weights the clustering coefficient of each vertex equally even though the number of possible connections of the neighbours of a vertex $v$ is $\binom{\deg(v)}{2}$. This can lead to the clustering coefficient of the graph being skewed by the clustering coefficients of the smaller degree vertices.  

Alternatively, to avoid the potential skew from smaller degree vertices, the clustering coefficient can be defined to weight every potential connection of neighbours equally. This leads to the \emph{global clustering coefficient} for graphs having at least one vertex of degree $2$ being defined as 
\begin{eqnarray}
C'(G) &=&  \frac { \sum_{v \in V, \deg(v) \geq 2} |\{ xy \in E : x,y \in N(v) \}| }{ \sum_{v \in V, \deg(v) \geq 2} {|N(v)| \choose 2} } \\
&=& \frac { 3 \times \text{number of triangles}}{ \text{number of pairs of adjacent edges}}.
\end{eqnarray}
This clustering coefficient is often thought of as the tendency of the graph to ``close'' triangles. 

Suppose that $G$ is a graph with at least one vertex of degree at least 2. Then $C(G)$ and $C'(G)$ are equal to  $1$ exactly when $G$ consists of the disjoint union of complete graphs. Both clustering coefficients are $0$ exactly when $G$ is bipartite; this is a bit of a shortcoming of the definitions since the complete bipartite graph $K_{n,n}$ has slightly more than half the edges of the complete graph, but has a clustering coefficent of $0$.   

A graph is considered {\em small-world} if its average local clustering coefficient $C(G)$ is significantly higher than a random graph constructed on the same vertex set, and if the graph has approximately the same mean shortest path length as its corresponding random graph.

There is no canonical way to generalize the idea of a graph clustering coefficient to hypergraphs, however, there have been a number of proposals: \cite{AR}, \cite{BE}, \cite{GLL}, \cite{LMV}, and \cite{ZN}. For this paper we will focus on the following definitions of Zhou and Nakhleh \cite{ZN} which generalize the idea of comparing the number of triangles to pairs of adjacent edges. Let $H = (V,E)$ be a hypergraph. If $v \in V$ and $e_i,e_j \in E$, let $M(v)$ denote the set of edges that contain $v$, let $N(v)$ denote the neighbours of $v$, and let $D_{ij} = e_i \setminus e_j = e_i \setminus (e_i \cap e_j)$. Local and global clustering coefficients on $H$ are then defined as 
\[
HC_{local}(v) =  \begin{cases}
\frac{1}{\binom{|M(v)|}{2}} \sum_{e_i,e_j \in M(v)} EO(e_i,e_j) & \text{ if } |M(v)| \ge 2 \\
 0  & \text{ if } |M(v)| \leq 1,
\end{cases}
\] 
and 
\[
HC_{global}(H) =  \begin{cases}
\frac{1}{|\mathfrak{I}|} \sum_{\{e_i,e_j\} \in \mathfrak{I}} EO(e_i,e_j) & \text{ if } \mathfrak{I} \neq \emptyset \\
 0  & \text{ if } \mathfrak{I} = \emptyset,
\end{cases}
\] 
where $\mathfrak{I} = \{ \{e_i,e_j\} : e_i,e_j \in E, e_i \neq e_j, \text{ and } e_i \cap e_j \neq \emptyset\}$ are the pairs of intersecting edges, and $EO(e_i,e_j)$, the \textit{extra overlap} of a pair of edges, is defined as 
\[
EO(e_i,e_j) = \frac{|N(D_{ij}) \cap D_{ji}| + |N(D_{ji}) \cap D_{ij}|}{|D_{ij}| + |D_{ji}|} 
\]
(the proportion of the vertices in exactly one of the edges that are neighbours of vertices in only the other edge).

Note that if $H$ is a graph then $HC_{global}(H) = C'(H)$ and $HC_{local}(v) = C(v)$. This follows from the fact that if $e_i = \{u,v\}$ and $e_j = \{u,w\}$ then the extra overlap $EO(e_i,e_j) = 1$ if $\{v,w\} \in E$ and $0$ otherwise.

We calculated local and global clustering coefficients for the four hypergraphs derived from real-world datasets and over a sample of random hypergraphs generated with the same edge distribution. Figure~\ref{fig:real_v_model_lcc} presents the distribution of local clustering coefficients for $\mathscr{E}_5$ and its corresponding model. This is illustrative of the results in general, and so we do not include other figures.  In fact, in the case of the coauthorship network, the results were even more glaringly dissimilar with the models having more than half the vertices with a local clustering coefficient of zero, while most of the vertices in the real-world networks were non-zero.  Specifically, the number of non-zero local clustering coefficients in $\mathscr{E}_5$ is 126844.  This is 7836 standard deviations above the average for our generative model with maximum edge size 5.

\begin{figure}[htbp]
\begin{center}
\includegraphics[width=5in]{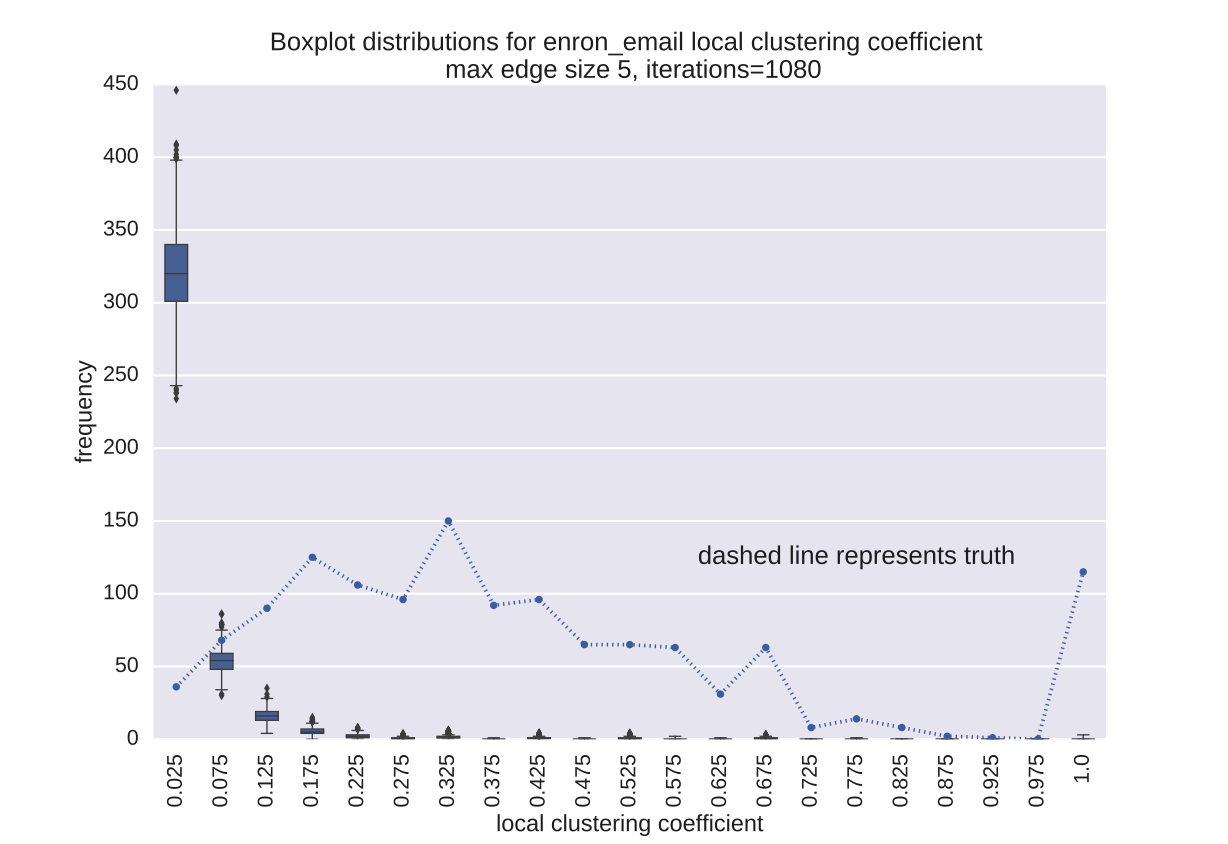} 
\caption{Local clustering coefficient of $\mathscr{E}_5$ and corresponding model hypergraphs}
\label{fig:real_v_model_lcc}
\end{center}
\end{figure}

Table~\ref{tab:clustering} presents the global clustering coefficents of the four real-world hypergraphs, and the mean over a sample of random hypergraphs generated with the same edge distributions.  

\begin{table}[h]
\begin{center}
\begin{tabular}{|c|c|c|c|c|} \hline
&$\mathscr{D}_4$ & $\mathscr{D}_5$ & $\mathscr{E}_4$ & $\mathscr{E}_5$\\ \hline
networks we examined & $0.0550$ & $0.0687$ & $0.1048$ & $0.1382$ \\ \hline
random hypergraphs & $8.01 \cdot 10^{-6}$ & $1.17 \cdot 10^{-5}$ & $0.00224$ & $0.00308$\\ \hline
\end{tabular}
\end{center}
\caption{Global clustering coefficient for networks we examined and corresponding model hypergraphs}
\label{tab:clustering}
\end{table}

Note that the true global clustering coefficient for $\mathscr{E}_5$ is approximately 41653 standard deviations from the mean of our model with max edge size 5.  Not so glaring, but still significant, the true global clustering coefficient for $\mathscr{D}_5$ is approximately 474 standard deviations above the mean of our model with max edge size of 5. 

Both the local and global clustering coefficient results indicate what we expect; the model assumption of edge independence is not reflective of real-world networks.  We all know researchers do not select coauthors at random!

\section{Conclusions and future work}

The ultimate goal of this work is to develop a reasonable model for complex networks using hypergraphs.  While there are many models using graphs -- including  classic ones such as the binomial random graph ($\G(n,p)$), random $d$-regular graphs, and the preferential attachment model, as well as spatial ones such as random geometric graphs and the spatial preferential attachment model -- there are very few using hypergraphs.  The model we proposed is a generalization of $\G(n,p)$ and thus, as we observed, does not capture several important features of many real-world networks.  However, it does allow us to identify some structures inherent in many of these networks, in particular, the non-independence of hyperedges.  It also allowed us to illustrate that some questions posed about hypergraphs cannot be addressed by looking at the 2-section.


\printbibliography

\end{document}